\documentclass[11pt,leqno,a4paper,english]{amsart}
\usepackage[utf8]{inputenc}
\usepackage[T1]{fontenc}
\usepackage[english]{babel}
\usepackage[a4paper, total={145mm, 255mm}, centering]{geometry}
\usepackage{enumerate}

\numberwithin{equation}{section}
\usepackage{amsmath, amssymb, amsthm}
\usepackage{mathtools}
\usepackage{mathrsfs}
\usepackage{booktabs,array}
\usepackage{xcolor}
\usepackage{cite}
\usepackage{hyperref}
\hypersetup{
colorlinks=true,
linkcolor=blue,
filecolor=magenta,
urlcolor=cyan,
citecolor=red
}
\newtheorem{theorem}{Theorem}[section]
\newtheorem{proposition}[theorem]{Proposition}
\newtheorem{lemma}[theorem]{Lemma}

\newtheorem{conjecture}[theorem]{Conjecture}
\theoremstyle{definition}
\newtheorem{definition}[theorem]{Definition}
\newtheorem{remark}[theorem]{Remark}

\newcommand{\xR}{\mathbb{R}}
\newcommand{\xT}{\mathbb{T}}
\newcommand{\xZ}{\mathbb{Z}}
\newcommand{\xN}{\mathbb{N}}
\newcommand{\xS}{\mathbb{S}}

\DeclareMathOperator{\Id}{Id}
\DeclareMathOperator{\diff}{d}
\DeclareMathOperator{\cn}{div}
\DeclareMathOperator{\curl}{curl}

\DeclareMathOperator{\supp}{supp}
\DeclareMathOperator{\diam}{diam}

\DeclareMathOperator{\tr}{tr}

\newcommand{\defn}{\coloneqq}
\def\la{\left\lvert}
\def\ra{\right\rvert}
\def\lA{\left\lVert}
\def\blA{\big\lVert}
\def\rA{\right\rVert}
\def\brA{\big\rVert}
\def\bla{\big\lvert}
\def\bra{\big\rvert}
\def\init{\textrm{init}}
\def\lin{\textrm{lin}}
\def\fract{\frac{\diff}{\dt}}
\def\dr{\diff \! r}
\def\dt{\diff \! t}

\def\dz{\diff \! z}
\def\e{\eqref}
\def\be{\begin{equation}}
\def\ee{\end{equation}}
\def\eps{\varepsilon}
\def\dx{\diff \! x}
\def\dxi{\diff \! \xi}
\def\dtheta{\diff \! \theta}
\def\dsigma{\diff \! \sigma}
\def\dy{\diff \! y}
\def\dtau{\diff \! \tau}
\def\les{\lesssim}

\newcommand{\av}[1]{\langle #1\rangle}
\newcommand{\Fro}{\mathrm{F}}

\title{Generic small-scale creation in the two-dimensional Euler equation}
\author{Thomas Alazard}
\address{CNRS - Centre de Mathématiques Laurent Schwartz, {\'E}cole Polytechnique,
Institut Polytechnique de Paris}
\email{thomas.alazard@polytechnique.edu}
\author{Ayman Rimah Said}
\address{Laboratoire de Mathématiques de Reims (LMR-CNRS UMR9008), Reims, France}
\email{ayman.said@cnrs.fr}
\date{}

\begin{document}

\begin{abstract}
The Cauchy problem for the two-dimensional incompressible Euler equation is globally well-posed for smooth
initial data. We prove that, for all
$p>2$ and for a dense $G_\delta$ set of initial vorticities in $W^{1,p}$, the solutions lose 
regularity in infinite time, in the sense 
that the $L^p$ norm of the gradient of the vorticity is
unbounded. This confirms a long-standing conjecture of Yudovich. 
\end{abstract}

\maketitle

\section{Introduction}
\label{S1}

The global well-posedness of the Cauchy problem for the two-dimensional incompressible Euler equations for smooth initial data is rooted in the transport 
of the vorticity, denoted by~$\omega$, by the flow $\Phi_t\colon\xR^2\to\xR^2$ of the velocity field: 
\begin{equation}\label{E11}
\omega(t, \Phi_t(x))=\omega_0(x).
\end{equation}
Since $\det(D\Phi_t)=1$, it follows from~\e{E11} that all the $L^p$-norms, $p \in [1, \infty]$, are conserved:
\begin{equation}\label{E16}
\lA \omega(t, \cdot) \rA_{L^p}=\lA \omega_0 \rA_{L^p}.
\end{equation}
However,~\e{E11} is also believed to drive
a generic growth of higher-order norms of the solution (see the discussion in \S\ref{S12}).
The main purpose of this paper is to establish
this fact. We denote by $W^{1,p}_0(\xT^2)$ and $H^s_0(\xT^2)$ the 
subspaces of $W^{1,p}(\xT^2)$ and $H^s(\xT^2)$ consisting of functions 
with vanishing mean value (the classical notations are recalled below in Section~\ref{S:notation}).

\begin{theorem}\label{T11}
Let $p > 2$ be a real number. The set of initial data
$\omega_0\in W^{1,p}_0(\xT^2)$
such that the corresponding unique solution of the two-dimensional incompressible Euler equations satisfies
$$
\sup_{t\geq 0} \lA \nabla\omega(t,\cdot) \rA_{L^{p}(\xT^2)}=+\infty,
$$
is a dense $G_\delta$ set in $W^{1,p}_0(\xT^2)$.
\end{theorem}
\begin{remark}\label{R:Hs}
$i)$ The assumption $p > 2$ is the minimal assumption
in terms of Sobolev embedding to ensure that the velocity belongs to $C^1(\xT^2)$, and hence that $\Phi_t$ is well-defined.

$ii)$
 Our proof also yields several variants in spaces of smoother functions. For instance, the same result holds when the space $W^{1,p}_0(\xT^2)$ is replaced
by $C^1_0(\xT^2)$, the space of $C^1$ functions with zero mean. 
We will also prove that, for
all real $s>1$, the set of initial data 
$\omega_0\in H^s_0(\xT^2)$ such that the corresponding solution satisfies 
$$
\sup_{t\geq 0} \lA \omega(t, \cdot) \rA_{H^s(\xT^2)}=+\infty
$$
is a dense $G_\delta$ set in $H^s_0(\xT^2)$. 
\end{remark}

Notice that Theorem~\ref{T11} holds in a domain
without boundary 
and without imposing any symmetry constraints.
This contrasts with previous results, 
such as those by Yudovich~\cite{zbMATH01620965}
and Morgulis, Shnirelman and Yudovich~\cite{zbMATH05308906},
where the analysis relies on boundary 
effects (specifically, fluid particles approaching stagnation points
on the boundary). Similarly, it differs from the work
of Elgindi, Murray and Said~\cite{zbMATH08109714},
which uses specific symmetry classes where the points on the axis of symmetry are naturally stagnation points. 

We also establish generic small-scale creation 
for solutions with initial data in $H^{s}(\xR^2)$ instead of $H^s(\xT^2)$.
We will see that this case is simpler since, through the scaling symmetries
of the Euler equations, spatial infinity can effectively act as a stagnation point
(at least from the perspective of the Lagrangian flow, as observed by Shankar
in~\cite{zbMATH06654624}). For $s>1$, we work in the Banach spaces 
\be\label{n103}
X^s(\xR^2)\defn H^s(\xR^2)\cap\dot H^{-1}(\xR^2), 
\qquad
\lA \cdot \rA_{X^{s}}\defn \lA \cdot\rA_{H^s}+\lA \cdot\rA_{\dot H^{-1}}. 
\ee
Notice that $\lA \omega\rA_{X^s} =\lA \nabla u\rA_{H^s}+\lA u\rA_{L^2}$, where $u$ is the velocity. 

\begin{theorem}\label{T:R2}
Let $s> 1$ be a real number. The set of initial data $\omega_0\in X^s(\xR^2)$
such that the corresponding unique solution of the two-dimensional incompressible Euler equations satisfies
$$
\sup_{t\geq 0} \lA \omega(t,\cdot) \rA_{H^{s}(\xR^2)}=+\infty,
$$
is a dense $G_{\delta}$ set in $X^s(\xR^2)$.
\end{theorem}

\begin{remark}\label{R:Ys} 
One can prove that the same result holds in the space 
$$
Y^s\defn \{\omega\in H^s(\xR^2)\cap L^1(\xR^2):\int_{\xR^2}\omega \dx=0\}. 
$$
For simplicity, we only consider $X^s(\xR^2)$.
\end{remark}

In the remainder of this introduction, we discuss related results and outline the proof.

\subsection{The Yudovich Conjecture}
\label{S12}
Numerical and physical observations show that loss of smoothness is an inherent feature of turbulent behaviour. Such phenomena were conjectured by Yudovich to be essentially ubiquitous.
%The conservation of the $L^p$ %norms leaves open the possibility that the derivatives
%of vorticity may grow.
%This is the object of the following

%
\begin{conjecture}[Yudovich (1974), \cite{yudovic1974loss,zbMATH01620965}, quote from \cite{zbMATH05308906}]
\label{C11}
There is a ``substantial set'' of inviscid incompressible flows whose
vorticity gradients grow without bound.
This set is dense enough to cause a loss of smoothness for arbitrarily
small perturbations of any steady flow.
\end{conjecture}

\begin{remark}
This conjecture has two components. The first one concerns the size of the set of initial data for which vorticity gradients grow without bound 
and the second one concerns the regularity in which the growth is measured. Regarding the size, our results yield a dense $G_\delta$ set in each of the spaces of
Theorems~\ref{T11} and~\ref{T:R2} and of Remark~\ref{R:Hs}. Regarding regularity, Theorem \ref{T11}
concerns precisely the vorticity gradient, as in Conjecture~\ref{C11}.
In addition, we show that growth occurs at arbitrarily 
high finite regularity. This distinguishes this phenomenon from instabilities 
which may be facilitated by a limited regularity, as in 
the $C^{1,\alpha}$ framework used in the pioneering works of 
Zlato\v{s}~\cite{Zlatos2015} and Elgindi \cite{zbMATH07441733}.
Our theorems establish instability in all the spaces $W^{1,p}$ with $p>2$ and $H^s$ with $s>1$. 
\end{remark}
This growth is expected to come from \emph{vorticity filamentation} (we refer to the recent review
articles \cite{zbMATH07803735,zbMATH07658194} for detailed presentations).
The flow stretches and folds initially smooth patches of vorticity into 
thinner and thinner filaments, and since the magnitude of $\omega$ is preserved along trajectories (see~\e{E11}),
the gradient across these filaments grows like the inverse of their thickness. This
deformation is a transfer to high frequencies, as is characteristic of 
turbulent behavior. 

Historically, the first examples of unbounded growth of the vorticity gradient were constructed by Bahouri and 
Chemin~\cite{zbMATH00687668} for singular initial data, namely
vortex patches with corners. For smooth solutions, Kiselev and \v{S}ver\'{a}k~\cite{KiselevSverak2014} proved
the existence of double-exponential growth on the disk,
a result extended by Xu~\cite{zbMATH06569430} to other
domains. Zlato\v{s}~\cite{Zlatos2015} then proved exponential
growth on the torus for solutions 
with limited smoothness and,
more  recently~\cite{Zlatos2025}, double-exponential
growth for smooth solutions on the half-plane. 
Even setting aside the question of genericity,
constructing solutions with superlinear growth in a domain 
without boundary remains difficult, with the exception of the work
of Denisov~\cite{zbMATH06393811}. By contrast, linear growth near steady states and traveling
waves is now well understood, thanks to the works of Nadirashvili \cite{zbMATH00022841},
Choi and Jeong~\cite{zbMATH07488949}, Drivas, Elgindi
and~Jeong~\cite{zbMATH07921933}, Bedrossian and Masmoudi~\cite{bedrossian2015inviscid},
Masmoudi and Zhao~\cite{masmoudi2024nonlinear}, and Ionescu and Jia~\cite{ionescu2020inviscid,ionescu2023nonlinear}.

\subsection{Strategy of the proof and plan of the paper}\label{S103}

The Eulerian instability in
Theorems~\ref{T11} and~\ref{T:R2} will be deduced from the following
Lagrangian instability result.

\begin{theorem}\label{T103}  
Let $s>1$ be a real number and assume that  either $M=\xT^2$ or $M=\xR^{2}$. For all 
non-zero initial data $\omega_0\in H^s(M)$ with
zero integral, and with compact support when $M=\xR^{2}$, the flow $\Phi_t$ of the corresponding
solution of the Euler equations satisfies 
\begin{equation}\label{n104}
\sup_{t\ge 0}\lA D\Phi_t\rA_{L^\infty(M)}=+\infty.
\end{equation}
\end{theorem}

This result is proved in Theorem~\ref{T105} on $\xT^2$ and in Proposition~\ref{P129} on $\xR^2$.
Before describing its proof, let us first explain how it implies the generic Eulerian instability.

\subsubsection*{Lagrangian to Eulerian transfer of instability (Sections~\ref{S:Koch} and~\ref{S507})}
To deduce the growth of the vorticity from~\e{n104},
we adapt an argument of Koch~\cite{zbMATH01801548} (see also Bourgain--Li~\cite{BourgainLi2015}, as discussed in Section~\ref{S:Koch}). The argument requires a perturbation of the vorticity to ensure it captures the Lagrangian growth, such a perturbation is necessary in the case of a steady vorticity for example. We first consider the instability in $H^{s}$ 
(Remark~\ref{R:Hs} and Theorem~\ref{T:R2}), and then explain the simpler case of Theorem~\ref{T11}. Write $\varphi_t= \Phi_t^{-1}$,
with $\lA D\varphi_t\rA_{L^\infty}=\lA D\Phi_t\rA_{L^\infty}$ as
$\det(D\Phi_t)=1$. Assume that $ D\varphi_T(y_0)$ 
has a large norm $\Lambda$, and choose $e\in \xS$ with $\la  D\varphi_T(y_0)^\top e\ra=\Lambda$. At time $t=0$, we insert a perturbation $\eta_{\init}$
with zero integral, localized at scale $\delta$ near $\varphi_T(y_0)$, oscillating 
in the direction $e$ with a frequency proportional to $1/\delta$, and with a small amplitude 
proportional to $\delta^{s-1}$. For integer $s$, Faà
di Bruno's formula applied to the $s$-th derivative of $\eta_{\init}\circ \varphi_T$
shows that this scaling isolates the term where all the derivatives act
on the profile $\eta_{\init}$, which produces $s$
factors of $D\varphi_T$. To control 
the change in the inverse flow caused by such a high-frequency 
perturbation, 
we use the fact that the velocity
is one derivative more regular than the vorticity. 
The perturbation used for Theorem \ref{T11} is simpler, since only one
derivative is involved and no oscillation is needed. It is a localized linear
profile of amplitude proportional to $\delta^{1-2/p}$, with gradient directed along $e$, and bounded in $W^{1,p}$ 
and in $H^{2-2/p}$ independently of $\delta$. The chain rule 
$\nabla(\eta\circ\varphi_T)=(D\varphi_T)^\top(\nabla\eta)\circ\varphi_T$ 
then gives a factor $\Lambda$ in the $L^p$ norm of the gradient, while the stability estimates in 
$H^{2-2/p}$ (recall that $2-2/p>1$) control the change of the inverse flow (see
Section~\ref{S507}).

\subsubsection*{A Baire argument}
The previous steps show that, for each finite threshold, the set of initial 
data whose solution exceeds this threshold at some 
time, in $ H^{s}$ or in $W^{1,p}$, is dense, and this set is open by continuous 
dependence at a fixed time. Sections~\ref{S:Baire} and~\ref{S507} then conclude the proof
by Baire's theorem, following the approach used by Hani~\cite{zbMATH06303378}
for dispersive equations. 

\subsubsection*{Virial identities (Sections \ref{S104}, \ref{sec: shankar} and~\ref{S473})} 
To prove \e{n104}, we seek functionals
whose growth would contradict a uniform bound 
on $D\Phi_t$, and for this purpose we study how the flow deforms a reference vector field. 

Given a smooth vector field $B=B(x)$ on $\xR^2$ or $\xT^2$, we consider the pull-back
$$
\Phi_t^*B=(D\Phi_t)^{-1}B\circ\Phi_t.
$$
Differentiating in time, one finds
$$
\partial_t\left(\Phi_t^*B\right) 
=%\Phi_t^*(\partial_tB_t)+
\Phi_t^*[u_t,B], 
$$
where $[u_t,B]=(DB)u_t-(Du_t)B$ is the Lie bracket. Since the flow 
preserves area, we also have, for every smooth scalar function $f$, 
$$
\Phi_t^*\nabla^\perp f=\nabla^\perp(f\circ\Phi_t).  
$$

On $\xR^2$, the vector field $B(x)=x$
generates dilations. This leads us to compare
$B$ with $\Phi_t^*B$ and to introduce
$$
V_t=\Phi_t^*B-B=(D\Phi_t)^{-1}\Phi_t-x. 
$$
Since $B$ and $\Phi_t^*B$ both have divergence two, we have $\cn V_t=0$,
and one can write $V_t=\nabla^\perp\Theta_t$. The functional introduced by
Shankar~\cite{zbMATH06654624} can then  be written as 
$$
G(t)\defn \int_{\xR^2}u_0\cdot V_t\dx=-\int_{\xR^{2}}\omega_0\Theta_t\dx.
$$
For compactly supported data with zero integral, we will prove that
$$
G'(t)=2\lA u_0\rA_{L^2}^2,
$$
so that $G$ grows linearly in time. On the other hand, if $D\Phi_t$ remained uniformly
bounded, energy conservation and the decay at infinity would imply that $G(t)=o(t)$. This contradiction proves
\e{n104} on $\xR^2$. 

Our proof of the identity 
$G'(t)=2\lA u_0\rA_{L^2}^2$ starts from the equation
$$
\partial_tV_ t=\nabla^\perp(\widetilde\psi_t\circ\Phi_t),  
\qquad  
\widetilde\psi_t(y)=2\psi_t(y)-y\cdot\nabla\psi_t(y).
$$
Notice that $\widetilde\psi_t=\partial_\lambda\psi^\lambda_t\big|_{\lambda=1}$,
where $\psi_t^\lambda(y)=\lambda^2\psi_t(y/\lambda)$ 
is the rescaling of $\psi_t$. Pairing the equation for $V_t$ with the initial 
velocity and using the transport of vorticity, we obtain 
$$
G'(t)=-\int_{\xR^2}\widetilde\psi_t\,\omega(t)\dy,
$$
and a Rellich identity~\cite{Rellich1940} then yields $G'(t)= 2\lA u(t)\rA_{L^{2}}^2 =2\lA u_0\rA_{L^2}^2$.

This approach based on vector fields goes beyond scaling arguments. Indeed, for a stationary reference field $B$ with constant
divergence  $c$, the observable $V_t=\Phi_t^*B-B$ satisfies 
the same evolution equation with $\widetilde\psi=c\psi-B\cdot\nabla\psi$, so that 
one can choose $B$ according to the geometry and
compute the production by a Rellich
identity. This approach, in the spirit of~\cite{AlazardZuily2024}, allows to extend our analysis to bounded domains and
the cylinder $\xT\times\xR $ for instance.

\subsubsection*{From the paradifferential observable to the polar angle}  

On $\xT^2$, however, dilations do not preserve periodicity, and a periodic
vector field cannot have non-zero constant divergence (the integral of the divergence vanishes),
so that the reference field $B(x)=x$ has no counterpart on
$\xT^2$. The formula $G=-\int\omega_0\Theta_t$ used
on $\xR^{2}$ nevertheless suggests looking for a scalar observable of the deformation,
to be paired with the initial vorticity, and our earlier work led us to choose an angle. 

In~\cite{AlazardSaid-arXiv}, following
the paradifferential framework of Shnirelman~\cite{zbMATH01026417}
and the second author~\cite{said2024small}, we considered 
$$
X_t=T_{(D\Phi_t)^{-1}}(\Phi_t-\Id), 
$$
where $T_a$ denotes Bony's paraproduct. The inverse differential 
also appears in the observable $V_t=\Phi_t^*B-B$ introduced
above, but here the paraproduct extracts the
high-frequency part of the displacement and
produces a periodic unknown, solution of
$$
\partial_tX_t=\nabla^\perp\left((\Phi_t)^{*,\mathrm{para}}\psi_t\right)+R_t,
$$
where $R_t$ is a smoothing remainder and $(\Phi_t)^{*,\mathrm{para}}\psi_t$ denotes
Alinhac's paracomposition of $\psi_t$ with $\Phi_t$ 
\cite{zbMATH03925354,zbMATH03960343,zbMATH07682221}.
The equation for $V$ can be viewed as a differential analogue of the evolution equation for $X$. 

To detect growth, in~\cite{AlazardSaid-arXiv} we tested 
$\curl X$ against the high-frequency part of a cusp function
of the initial vorticity. More precisely, we used 
$$
y_\varepsilon(t) =\int_{\xT^2}\curl X(t)\, P_\varepsilon(\omega_0-m)^\alpha\dx,
\qquad m=\min_{\xT^2}\omega_0,
$$
where $P_\varepsilon$ cuts off frequencies of size smaller than $\varepsilon^{-1}$.

This observable was used to prove $H^s$ growth of 
the flow near non-degenerate extrema of the vorticity, where the fluid has
a persistent differential rotation. The geometric picture is as follows. Consider a smooth vorticity 
with one maximum and one minimum, both non-degenerate, and a 
material line joining them. As time goes on, one expects this line to wind 
more and more around the two extrema, unless the extrema become degenerate.
In both cases, one expects \e{n104} to hold. The
difficulty in making this picture quantitative comes from the nonlocality of the Biot--Savart law. 

To relate this picture to $X$, write the polar decomposition 
$$
D\Phi_t=R_{\theta_t}S_t,\qquad S_t=\left((D\Phi_t)^\top D\Phi_t\right)^{1/2},
$$
and choose the real-valued lift of $\theta_t$ continuously in time, with $\theta_0=0$. For smooth 
solutions, an expansion at $t=0$ shows that 
$$
\curl X_t = 2T_1\theta_t+O(t^{2}),
$$
where $T_1$ is the high-frequency cutoff associated with
the paraproduct, and this comparison, together with the planar 
virial, leads us to introduce
$$
\mathscr A(t)=\av{\omega_0\theta_t},
$$
where $\av{\cdot}$ denotes the spatial average on 
$\xT^2$. Unlike $y_\varepsilon$, the functional $\mathscr A$ involves $\omega_0$ itself (which is closer to the functionals in \cite{zbMATH01026417,said2024small}), without a cusp function or a
frequency cutoff, and uniform rotations do not contribute to it since $\av{\omega_0}=0$ (it measures the relative 
rotation of the regions of positive and negative initial vorticity). 

The new ingredient is an exact identity expressing $\mathscr A'$ as the
spatial average of a sum of squares, from which we deduce that, for all $s>1$ and every non-zero
initial datum, without conditions on the extrema or the Fourier coefficients of $\omega_0$,
$$
\mathscr A'(t)\ge
\frac{\av{\omega_0^2}}{4\lA D\Phi_t\rA_{L^\infty}^2}
$$
(see Theorem~\ref{T104}). It then remains to control the accumulated angle (the size
of $D\Phi_t $ does not directly bound the number of turns). Using 
periodicity to reconstruct this angle from material lines, we prove that a 
uniform bound on $\lA D\Phi_t\rA_{L^\infty}$ 
implies $\la\mathscr A(t)\ra\le C(1+\log(1+t))$, in contradiction with the
linear growth given by the lower bound for $\mathscr A'$. This proves \e{n104}
on $\xT^2$.

\subsection*{Plan of the paper.}
We recall the standard well posedness theory needed here in Section \ref{sec: prelim}, Section~\ref{S104} introduces the Lagrangian
observables and the virial identities, and 
Sections~\ref{sec: shankar} and~\ref{S473} prove \e{n104}
on $\xR^{2}$ and on $\xT^2$. Sections~\ref{S:Koch} and~\ref{S:Baire} then establish the transfer to
the vorticity and the genericity statements 
in  $H^s$, and Section~\ref{S507} proves Theorem~\ref{T11}. The results on the
Cauchy problem and the flow map used in these arguments are gathered in Appendix~\ref{S508}. 

\medskip

\subsection*{Acknowledgements.} 
The authors warmly thank Patrick Gérard and Herbert Koch for very interesting discussions.
We are also grateful to Tarek M.\ Elgindi for bringing to our
attention Shankar's work \cite{zbMATH06654624} and many helpful discussions.

T.A.\ is partially
supported by the ANR projects BOURGEONS (grant ANR-23-CE40-0014-01)
and NO-LIMIT (grant ANR-25-CE40-1380) and A.R.S.\ is partially supported by
the ANR project SMASH (grant ANR-25-CE40-4532), of the French National Research Agency (ANR).

\section{Preliminaries}\label{sec: prelim}

\subsection{The equations}\label{S11}

The two-dimensional incompressible Euler equation reads
\begin{equation}\label{E12}
\left\{
\begin{aligned}
&\frac{\partial u}{\partial t}+u\cdot\nabla u+\nabla p=0, \\
&\cn u=0,
\end{aligned}
\right.
\end{equation}
where $u(t,x)=(u^1,u^2)(t,x)$ and $p(t,x)$ denote the velocity and the pressure, respectively. We consider bi-periodic functions $f\colon\xR^2\to\xR$, satisfying
$$
f(x_1+2\pi ,x_2)=f(x_1,x_2+2\pi)=f(x_1, x_2). 
$$
which we identify with functions defined on the $2$-torus $\xT^2\defn ( \xR /2\pi\xZ)^{2}$. 
%, 
% whose Sobolev spaces are denoted by
% $H^{s}(\xT^2)$ for $s\in\xR$. Recall that 
% $H^s_0(\xT^2)$ is the subspace of functions $f$ with mean value zero, $\int_{\xT^2}f(x)\dx=0$.

The vorticity is defined by
$\omega=\partial_1 u^2-\partial_2 u^1$. 
Taking the curl of the momentum equation yields a transport equation for the vorticity:
\begin{equation}\label{E13}
\frac{\partial  \omega}{\partial t}+u \cdot \nabla \omega=0.
\end{equation}
Introducing a stream function $\psi$ with $u=\nabla^\perp\psi\defn (-\partial_2\psi, \partial_1\psi)$,
we recover $u$ from $\omega$ by the Biot--Savart law (see~\e{BS}): 
\begin{equation}\label{E15}
u= \nabla^{\perp}\Delta^{-1}\omega. 
\end{equation}
%Thus, by \e{E13}, the vorticity is transported by the flow that it generates.

The trajectory $\Phi_t(x)$ starting at $x\in \xR^2$ solves 
$$
\frac{\diff}{\dt}\Phi_t(x)=u(t, \Phi_t(x)), \quad \text{with} \quad \Phi_0(x) = x.
$$
Along these trajectories, \e{E13} reduces to the identity \e{E11}, and
hence the conservation \e{E16} of the $L^p$ norms of the vorticity, for $p\in [1,\infty]$. 

\subsection{Notations and conventions}\label{S:notation}

\subsubsection*{Vectors and Matrices}
Vectors are identified with column matrices. We write
$$
\nabla^\perp f=J\nabla f, \qquad
J\defn\begin{pmatrix}0&-1\\1&0\end{pmatrix}.
$$

In a matrix $A=(a^i_j)$, $i$ indexes the rows (target) and $j$ the columns
(source). Thus the Jacobian matrix $Du$ of $u=(u^1, u^2)^\top$ is defined by 
$$
(Du)^i_j \defn \partial_j u^i \quad \implies \quad
Du=\begin{pmatrix} \partial_1 u^1 & \partial_2 u^1 \\ \partial_1 u^2 & \partial_2 u^2 \end{pmatrix}.
$$
The transpose of a matrix is denoted by the superscript ${}^\top$. 
%The components of the transposed Jacobian are $((Du)^\top)^i_j=\partial_i u^j$.
We write $(Du)^{-\top}=((Du)^{-1})^\top$ for the transpose of the inverse matrix. For 
a $2\times2$ matrix $P$, we denote by $\la P\ra$ 
its operator norm for the Euclidean norm of $\xR^2$, and by
$\la P\ra_{\Fro}^{2}=\tr(P^\top P)$ the square of the Frobenius norm of $P$.
When $\det P=1$, the two singular values of $P$ (square root of the eigenvalues of $P^\top P$) are of the form $\lambda$
and $\lambda^{-1}$ with $\lambda\ge 1$, so that 
\be\label{n105}
\la P^{-1}\ra=\la P\ra = \lambda\ge 1,\qquad \la P\ra_{\Fro}^2=\lambda^2+\lambda^{-2}\le \la P\ra^2+ 1 .
\ee
For two vectors $a,b\in\xR^2$, we write $a\wedge b=a_1b_2-a_2b_1$. 
The identity matrix is denoted by $I$ or $I_2$, and the identity map by $\Id$. 

\subsubsection*{Function Spaces}
For an index $s\in [0,+\infty)$, $H^s(\xT^2)$ consists 
of those functions $u$ in $L^2(\xT^2)$ for which the following norm is finite: 
$$
\lA u\rA_{H^{s}}\defn \biggl(\sum_{\xi\in \xZ^2} \big(1+\la \xi\ra^{2}\big)^s \la \hat{u}(\xi)\ra^2\biggr)^{1/2},
$$
and $H^s(\xR^2)$ is defined by the same
formula, with the sum replaced by an integral. We
write 
$\av f=(2\pi)^{-2}\int_{\xT^{2}}f\dx$ for the mean of a periodic function. 

For $k\in\xN$, we denote by $C^k(\xT^2)$ the usual space of functions on
$\xT^2$ whose derivatives of order $\le k$ are continuous.
For a non-integer $r>0$, $C^r(\xT^2)$ denotes the space of functions whose
derivatives up to order $[r]$ are bounded and uniformly H\"older continuous
with exponent $r-[r]$, where $[r]$ is the integer part of $r$.
The H\"{o}lder norm is defined as 
$$
\la u\ra_{C^r}\defn \la u\ra_{C^{[r]}}
+\sum_{|\alpha|=[r]}\sup_{x,y\in \xT^2, x \neq y}\frac{\la\partial^{\alpha}u(x)-\partial^{\alpha}u(y)\ra}{\la x-y\ra^{r-[r]}}.
$$
For $k\in\xN$ and $0<\gamma<1$, we also write 
$C^{k,\gamma}(\xT^2)=C^{k+\gamma}(\xT^2)$, and $C^\gamma=C^{0,\gamma}$. Recall that, 
for $s>1$, $H^s(\xT^{2})\subset C^{0,\gamma}(\xT^{2})$ for every 
$0<\gamma<\min(1,s-1)$. Similarly, for $p>2$, we have $W^{1,p}(\xT^2)\subset C^{0,1-2/p}(\xT^2)$. We denote by $C^1_0(\xT^2)$ 
the space of $C^1$ functions with zero mean.

\subsection{Well-posedness and Flow Map}
\label{S:classique}
We consider the two-dimensional incompressible Euler equations on $\xT^2$, with zero-mean vorticity, or on $\xR^2$.

\smallbreak
\noindent\textbf{Functional setting.}
For $k\in \xZ^2\setminus\{0\}$, the Fourier symbol of
the Biot--Savart operator $\nabla^\perp\Delta^{-1}$ is
\be\label{BS}
m(k)=-\frac{i k^\perp}{|k|^2}, \quad \text{where } k^\perp=(-k_2, k_1).
\ee
This linear operator is bounded from $H^s_0(\xT^2)$ into $H^{s+1}(\xT^2)$.
On $\xR^2$, the same symbol $m(\xi)$, $\xi\in \xR^2\setminus\{0\}$,
defines the Biot-Savart operator on the space $X^s(\xR^2)$ 
of vorticities with finite energy defined in \e{n103}.
This operator is bounded from $X^s(\xR^2)$ into 
$H^{s+ 1}(\xR^2)$, with $\lA u\rA_{L^2}= \lA \omega\rA_{\dot H^{-1}}$ 
(see Proposition~\ref{P203}). For a compactly 
supported vorticity $\omega\in L^1\cap L^\infty(\xR^2)$ with zero integral 
(such a vorticity belongs to $\dot H^{-1}(\xR^2)$),
the velocity is also given by the convolution with the Biot--Savart kernel
$$
K(z)=\frac{1}{2\pi}\frac{z^\perp}{\la z\ra^2}\cdot 
$$
To unify the notations, we set $ X^s(\xT^2)\defn H^s_0(\xT^2)$.

\smallbreak
\noindent\textbf{The Flow Map.}
Throughout this paper, we denote by $\Phi^{\omega_0}_t$ the flow map associated with the unique
solution generated by the initial data $\omega_0$, and by
$\varphi^{\omega_0}_t=(\Phi^{\omega_0}_t)^{-1}$ the inverse flow. 
The solution map $S(t)\colon\omega_0\mapsto 
\omega(t)$ is therefore given by 
$$
S(t)\omega_0 =\omega_0 \circ \varphi^{\omega_0}_t.
$$
On $\xT^2$, the flow of a zero-mean velocity field is understood through its lift to the plane,
satisfying $\Phi_t(x+2\pi k)= \Phi_t(x)+2\pi k$ for $k\in \xZ^2$, so that $\Phi_t-\Id$ and
$ D\Phi_t$ are periodic. Since $\cn u=0$, the flow preserves 
area, $\det D\Phi_t=1$. The identity
$\varphi_t\circ\Phi_t=\Id$ implies that $D\varphi_t(\Phi_t(x))=(D\Phi_t(x))^{-1}$, and combining this relation 
with~\e{n105}, we find that
\be\label{n106}
\lA D\varphi_t\rA_{L^\infty}=\lA D\Phi_t\rA_{L^\infty}\ge 1 .
\ee

\smallbreak
The following theorem gathers the classical results 
on the Cauchy problem which we use. Hereafter, $M$ 
denotes either $\xT^2$ or $\xR^2$, and $H^r_\sigma(M)$ denotes the space of 
divergence-free vector fields $u\in H^{r}(M)^2$, with 
zero mean when $M=\xT^{2}$. 

\begin{theorem}[classical results]\label{T:classique}
Let $s> 1$. For all $u_0\in H^{s+1}_\sigma(M)$,
there exists a unique solution $u\in C(\xR;H^{s+1}_\sigma(M))\cap C^1(\xR;H^s(M))$ of the Euler equations with $u(0)=u_0$.
This solution has the following properties: 
\begin{enumerate}
\item the vorticity satisfies $\omega\in C(\xR;X^s(M))$ and is given by 
$\omega(t)=\omega_0\circ\varphi_t$;  
\item the flow $\Phi_t$ and its inverse $\varphi_t$ are $C^{1}$
area-preserving diffeomorphisms, and $\Phi-\Id$, $\varphi-\Id$ belong to $C(\xR;H^{s+1}(M))$;
\item there exists a non-decreasing function $C_s$, depending only on  
$s$ and $M$, such that, for all $T>0$, 
$$
\sup_{\la t\ra\le T}\lA u(t)\rA_{H^{s+1}}\le C_s(T,\lA u_0\rA_{H^{s+1}});
$$
\item for all $T>0$, the solution map 
$$
u_0\longmapsto u\colon H^{s+1}_\sigma(M)\longrightarrow C([-T,T], H^{s+1}_\sigma(M))
$$
is continuous.
\end{enumerate}
\end{theorem}

The map $u\mapsto\curl u$ is an isomorphism from $H^{s +1}_\sigma(M)$ onto 
$X^s(M)$ (see Proposition~\ref{P203}).
Hence, for all $\omega_0\in X^{s}(M)$, there exists a unique solution with initial
vorticity $\omega_0$, and, for any fixed $t\in \xR$, the 
solution map $S(t)$ is continuous from $X^s(M)$ to $X^s(M)$. Moreover, the energy $\lA u\rA_{L^2}$ is 
conserved. For initial vorticities in 
$W^{1,p}_0(\xT^2)$ with $p>2$, or in $C^{1}_0(\xT^2)$, the solution to the Cauchy problem is recalled in Lemma~\ref{L450}. 

\section{Virial identities for the Euler flow}\label{S104}

In this section we introduce the two virial identities on which the proof relies. On 
$\xR^2$, we start from a conservation law discovered by
Shankar~\cite{zbMATH06654624}, while on $\xT^2$ we use instead the rotation 
of the differential of the flow (dilations do not preserve periodicity). 

\subsection{A variational virial}
We begin by recalling the variational interpretation of virial identities, following the presentation
of \cite{AlazardZuily2024} for water waves. 

Consider a particle of unit mass, with position $q$, moving in a potential $U$, so that $\ddot q=-\nabla U(q)$. 
Differentiating the function 
$G= q\cdot\dot{q}  $
and using the equation, we obtain 
\be\label{n401} 
G'= \la\dot q\ra^2+q\cdot\ddot q =\la \dot q\ra^2-q\cdot\nabla U(q).
\ee
If $U$ is homogeneous of degree $k$, then $q\cdot\nabla U(q) =kU(q)$ 
and $G'=2E_k-kE_p$, where $E_k$ and $E_p$ denote the kinetic and potential 
energies. If moreover $q$ and $\dot q$ stay bounded, then so does $G$, and
integrating in time, we deduce that
$$
 \frac1T\int_0^T(2E_k-kE_p)\dt\longrightarrow0. 
$$
This is the virial theorem of Clausius \cite{Clausius1870}, which asserts that, on average, the kinetic energy is $k/2$ times the potential energy. 

To construct a monotone quantity, we read 
\e{n401} as a variational identity. Replacing $(q, \dot q)$ by $(\lambda q, \lambda\dot q)$ in the Lagrangian
$\frac{1}{2}\la\dot q\ra^2-U(q)$ and differentiating at $\lambda= 1$, we recover 
the right-hand side of \e{n401}. This suggests looking for 
a deformation under which the Lagrangian is not 
invariant, but whose variation has a definite sign. The time derivative 
of the associated observable then has the same sign, because it is equal to
the variation of the Lagrangian along this 
deformation (see \cite{AlazardZuily2024}). 

\subsection{Lagrangian observables}\label{S105} 

For the Euler equations, we look for virial functionals which pair the Lagrangian 
velocity with a variation of the flow.
%We first explain how the incompressibility constraint eliminates the 
%pressure from their time derivative.

In Arnold's geodesic formulation~\cite{zbMATH03240308}, the configuration 
is the flow map $\Phi$, and the Lagrangian 
is the kinetic energy $\frac{1}{2}\int\la \partial_t\Phi\ra^2\dx$, subject to the
constraint $\det D\Phi=1$. We denote by $x$ the label of
a particle and by $y=\Phi(x)$ its position, and set $v=\partial_t\Phi = u\circ\Phi$. 
The equation of motion reads
\be\label{n402} 
\partial_tv=-(\nabla p)\circ\Phi. 
\ee
If $Y$ is a time-dependent variation of $\Phi $ such that $W=Y\circ\Phi^{-1}$
is divergence-free), then changing variables and integrating by parts, we obtain
\be\label{n403} 
\int(\nabla p)\circ\Phi\cdot Y \dx
=\int\nabla p\cdot W\dy=0, 
\ee
whenever the boundary terms vanish. Combined with \e{n402}, this implies that
$$
\frac{\diff}{\dt}\int v\cdot Y\dx 
=\int v\cdot\partial_tY\dx,
$$
and it remains to choose $Y$ so that $\int v\cdot\partial_tY\dx$ has a definite sign. 

We construct $Y$ from a stationary vector field $X$, by comparing $X$ with the pull-back of $X$ by the flow. 

\begin{definition}[Lagrangian observable]\label{D71}
Let $\Phi_t$ be an area-preserving flow and let $X=X(x)$ be a smooth vector field. The Lagrangian observable associated 
with $X$ is the vector field defined by
\be\label{E7:V}
V(t, x)\defn (D\Phi_t(x))^{-1}X(\Phi_t(x))-X(x). 
\ee
\end{definition}

By definition, $V+X=\Phi_t^*X$ is the pull-back of $X$, and the 
corresponding variation of the flow is 
$Y=D\Phi_t\,V= X\circ\Phi_t-D\Phi_tX$. To differentiate $V$, we will
use two elementary matrix identities: firstly, for all $2\times2$ matrices $N$, we have
\be\label{n404}
NJ + JN^\top =\tr(N)\,J. 
\ee
Moreover, if $\det N=1$, the Cayley--Hamilton relation
$N^{-1}=\tr(N)I_2-N$ combined with~\e{n404} implies that
\be\label{n405} 
N^{-1} J=JN^\top . 
\ee

The identity \e{n404} leads to the following expression for the Lie bracket.

\begin{lemma}\label{L72} 
Let $u=\nabla^\perp\psi$. For every smooth
vector field $X$, the Lie bracket $[u,X]=(DX)u-(Du)X$ satisfies 
\be\label{n406}
[u,X]=(\cn X)u-\nabla^\perp(X \cdot \nabla \psi). 
\ee
\end{lemma}

\begin{proof}
Directly from $u=J\nabla\psi$, we have $Du=J D^2\psi$, hence 
$$
(D^2\psi)X= (X\cdot\nabla)\nabla\psi =\nabla(X \cdot\nabla \psi)-(DX)^\top\nabla\psi.
$$
This implies that 
$$
[u,X]=(DX)u-(Du)X
=\left((DX)J+J(DX)^\top\right)\nabla \psi-J \nabla(X \cdot \nabla \psi). 
$$
Applying~\e{n404} with $N=DX$, we see that
the first term is $(\cn X)J\nabla\psi=(\cn X )u$, and \e{n406} follows. 
\end{proof}

We next use~\e{n405} to compute the pull-back of a perpendicular gradient. 

\begin{lemma}\label{L73}
For all diffeomorphisms $\Phi$ with $\det D\Phi=1$ and all 
differentiable scalar functions $F$, there holds 
\be\label{n434}
(D\Phi)^{-1}(\nabla^\perp F\circ \Phi )=\nabla^\perp(F \circ \Phi).
\ee
\end{lemma}

\begin{proof}
Applying \e{n405} with $N=D\Phi$, we get
$(D\Phi)^{-1}J=J(D\Phi)^\top$. Also, by the chain rule, 
$(D\Phi)^\top(\nabla F \circ \Phi)=\nabla(F \circ \Phi)$. This proves \e{n434}, since $\nabla^\perp=J\nabla$. 
\end{proof}

Let us now compute the divergence and the time derivative of the observable $V$. 
\begin{lemma}\label{L71}
Let $u=\nabla^\perp \psi $ solve the Euler
equations on $M$, with flow $\Phi_t$, and let $X$ be a smooth stationary vector field. 

$i)$ The Lagrangian observable satisfies, in the sense of distributions, 
\be\label{n435}
\cn V=(\cn X)\circ \Phi_t-\cn X,
\ee
and in particular $\cn V=0$ when $\cn X$ is constant.

$ii)$ If $\cn X=c$ is constant, then 
\be\label{n436}
\partial_t V = \nabla^\perp\big(\tilde\psi(t,\cdot)\circ\Phi_t\big), 
\qquad \tilde\psi\defn c\psi-X\cdot\nabla\psi . 
\ee
\end{lemma}
\begin{proof}
$i)$ Set $W(t,x)\defn (D\Phi_t(x))^{-1}X(\Phi_t(x))=V(t,x)+X(x)$.
Since $\det D\Phi_t=1$, we have 
$$
(D\Phi_t)^{-1}=\begin{pmatrix}\phantom{-}\partial_2 (\Phi_t)_2 & -\partial_2 (\Phi_t)_1 \\
-\partial_1 (\Phi_t)_2 & \phantom{-}\partial_1 (\Phi_t)_1 \end{pmatrix},
$$
whose columns are divergence-free in the sense of distributions. Setting 
$A=(D\Phi_t)^{-1}$ and using the product rule, we find that
\begin{align*}
\cn W&=\sum_j\Big( \sum_i \partial_i A_{ij}\Big) (X\circ\Phi_t)_j
+ \sum_{i,j} A_{ij}\,\partial_i (X\circ\Phi_t)_j\\
&=\tr\Big( (D\Phi_t)^{-1}\big((DX)\circ\Phi_t\big) D\Phi_t \Big)
=(\cn X)\circ\Phi_t , 
\end{align*}
where we used $D(X\circ\Phi_t)=((DX)\circ\Phi_t) D\Phi_t$. 
This implies~\e{n435}, since $V=W-X$.

$ii)$ Differentiating the equation $\partial_t\Phi_t=u(t, \Phi_t)$ 
with respect to $x$, we get $\partial_t D\Phi_t=(Du\circ\Phi_t)D\Phi_t$.
Consequently, $\partial_t((D\Phi_t)^{-1})=-(D\Phi_t)^{-1}(Du \circ \Phi_t)$. 
We also have $\partial_t\big(X\circ\Phi_t\big)=\big((DX)u\big)\circ\Phi_t$, because 
$X$ is stationary. By the product rule, 
$$
\partial_t V=(D\Phi_t)^{-1}\left((DX)u-(Du)X\right)\circ \Phi_t=(D\Phi_t)^{-1}([u,X]\circ \Phi_t). 
$$
Then~\e{n436} follows from Lemma~\ref{L73} 
applied with $F=\tilde\psi(t,\cdot)$, because Lemma~\ref{L72} (applied with 
$\cn X=c$) implies that $[u,X]=\nabla^\perp\tilde\psi$.
\end{proof}

\subsection{The dilation virial on \texorpdfstring{$\xR^2$}{R2}}\label{S106}

We use Shankar's conservation law~\cite{zbMATH06654624} in a
form involving the initial velocity: for $s>1$ and
$\omega_0\in H^s(\xR^2)$ with compact support and zero integral, we set
\be\label{n437}
V(t,x)=(D\Phi_t(x))^{-1}\Phi_t(x)-x,
\qquad G(t)=\int_{\xR^2}u_0\cdot V \dx .
\ee
%Denote by $E=\frac12\lA u_0\rA_{L^2}^2$ the conservedenergy. 
Then Shankar's identity reads 
\be\label{n394} 
G'(t)=2\lA u_0\rA_{L^2}^2 . 
\ee
Integrating from $V(0)=0$, we get $G(t)=2\lA u_0\rA_{L^2}^2t$, and this identity will imply
that $\lA D\Phi_t\rA_{L^\infty}$ cannot remain 
bounded in time when $\omega_0\neq0$ (see Proposition \ref{P129}). 

\subsection{The rotation virial on \texorpdfstring{$\xT^2$}{T2}}\label{S471}

On $\xT^2$, where dilations do not preserve periodicity, we 
measure the rotation of the differential of the flow by means of the polar decomposition 
$$
D\Phi_t=R_{\theta_t}S_t,\qquad S_t=\big((D\Phi_t)^\top 
D\Phi_t\big)^{1/2},\qquad 
R_\theta=\begin{pmatrix}\cos\theta&-\sin\theta\\ \sin\theta&\cos\theta\end{pmatrix},
$$
where the matrix $S_t$ is symmetric, positive definite, and has determinant one. We choose the
real-valued angle $\theta_t$ continuously in time,
with $\theta_0=0$, and set 
\be\label{n395} 
\mathscr A(t)\defn \av{\omega_0\,\theta_t} .
\ee
Here, as in the introduction, $\av{\cdot}$ denotes the 
spatial average on $\xT^2$. At $t=0$, we have $\partial_t\theta_t=\omega_0/2$, and therefore $\mathscr A'(0)=\frac{1}{2}\av{\omega_0^{2}}$. The following result 
asserts that $\mathscr A$ is non-decreasing for all times, and 
bounds $\mathscr A'$ from below.

\begin{theorem}[Rotation virial]\label{T104}
Let $s>1$, let $\omega_0\in H^s_0(\xT^2)$, let
$\Phi_t$ be the flow of the solution with 
initial vorticity $\omega_0$, and let $D\Phi_t=R_{\theta_t}S_t$
be the polar decomposition above. Then $\mathscr A(0)=0$ and 
\be\label{n396}
 \mathscr A'(t)
 =\Big\langle(\partial_t\theta_t)^2+\frac12\big\lvert S_t^{-1/2}(\partial_tS_t)S_t^{-1/2}\big\rvert_{\Fro}^2\Big\rangle . 
\ee
In particular we have 
\be\label{n397} 
 \mathscr A'(t)\ge \frac{\av{\omega_0^2}}{4\lA D\Phi_t\rA_{L^\infty}^2} .  
\ee
\end{theorem}

The proof is given in Section~\ref{S473}.

\section{Unbounded deformation on \texorpdfstring{$\xR^2$}{R2}}\label{sec: shankar}

In this section we prove that the differential of the flow map is
unbounded in time for all non-zero
compactly supported initial vorticity with zero
integral (see Proposition~\ref{P129}).

\subsection{Estimates on the velocity field}

We first recall a standard estimate for the velocity in terms of the 
conserved $L^1$ and $L^\infty$ norms of the vorticity. 

\begin{proposition}\label{P128} 
Let $\omega \in L^1(\xR^2) \cap L^\infty(\xR^2)$ 
and let $u$ be the associated velocity field. Then $u$ belongs to 
$L^\infty(\xR^2)$ and satisfies
\be\label{n40}
\lA u \rA_{L^\infty}\le C \lA \omega \rA_{L^1}^{1/2}\lA \omega \rA_{L^\infty}^{1/2}. 
\ee
\end{proposition}

Let us now estimate the velocity and the velocity gradient at infinity, using the zero-integral condition on the vorticity.

\begin{lemma}\label{lem: kernel decay}
Let $\omega \in L^{1}(\xR^2) \cap L^\infty(\xR^2)$ be compactly supported with 
$\int_{\xR^2}\omega(y) \dy =0$. Then, 
for all $y_c \in \xR^2$ and all $x\in \xR^2$ such that
$$
\la x-y_c\ra \ge 2 \la y - y_c\ra \quad\text{for all}\quad 
y \in \supp \omega,
$$
there holds 
\be\label{far}
\begin{aligned}
\la u(x)\ra &\le \frac{ C}{\la x- y_c\ra^2}\int_{\xR^2}\la y - y_c\ra \la \omega(y)\ra\dy \quad \text{and}\\
\la D u(x)\ra &\le \frac{C }{\la x-y_c \ra^3}\int_{\xR^{2}}\la y-y_c \ra \la \omega(y)\ra\dy.
\end{aligned}
\ee
In particular, $u\in L^{2}(\xR^2) $. 
\end{lemma}

\begin{proof}
Recall that we denote by $K(z)=\frac{1}{2\pi}\frac{z^\perp}{\la z\ra^2}$ the
Biot--Savart kernel. Fix $y_c \in \xR^2$. Using that $\int \omega= 0$, we write 
$$
u(x)=\frac{1}{2\pi}\int_{\supp \omega}\left(
\frac{(x-y)^{\perp}}{\la x-y \ra^2}-\frac{(x-y_c)^{\perp}}{\la x -y_c\ra^2}\right)
\omega(y) \dy.
$$
For $x$ as in the statement and $y\in \supp\omega$, the inequality
$\la x-y\ra\ge \la x-y_c\ra-\la y-y_c\ra\ge \tfrac12\la x-y_c\ra>0$ shows that $u$ is smooth near 
$x$ thus one can differentiate under the integral sign to compute $Du(x)$. 

Given $y\in \supp \omega$, we parametrize the segment joining $x-y_c$ 
and $x-y$ by $z(\theta)= (x-y_c) -\theta(y-y_c)$ for 
$\theta \in [0,1]$. By assumption on $x$, we have 
$$
\la z(\theta)\ra \ge \la x- y_c\ra-\theta \la y-y_c\ra
\ge \la x - y_c\ra-\frac{1}{2}\la x-y_c\ra=\frac{1}{2}\la x-y_c\ra. 
$$
The mean value theorem and the bound $\la \nabla K(z)\ra\le C_1 \la z \ra^{-2}$ imply that 
$$
\la K(x-y)- K(x-y_c) \ra\le \la y-y_c\ra \sup_{\theta \in [0,1]}\la \nabla K(z(\theta)) \ra
\le \frac{4C_1}{\la x -y_c \ra^2}\la y - y_c\ra.
$$
Integrating against $\la\omega(y)\ra$, we deduce the first estimate in~\e{far}:
$$
\la u(x)\ra\le \frac{4C_1}{\la x-y_c\ra^2}\int_{\xR^2}\la y-y_c\ra \la \omega(y)\ra \dy.
$$
The second estimate in~\e{far} for $Du$ is proved in the same way, with $K$ replaced by the kernel $H=\nabla K $, homogeneous of degree $-2$. 

Finally, \e{far} implies that $u \in L^2(\xR^2)$, 
because $u$ is bounded (see Proposition \ref{P128})
and $\la x\ra^{-2}$ is square integrable at infinity on $\xR^2$. 
\end{proof}

\subsection{The virial identity in Lagrangian coordinates}

We now apply Definition~\ref{D71} with $X(x)=x$ and prove
identity~\e{n394} of Section~\ref{S106} for $ s> 1$.
%In the next subsection, this identity gives the
%lower bound from which we deduce unbounded deformation of the flow.

Let $s>1$ and let $\omega_0\in H^s(\xR^2)$ have compact
support and zero integral. It follows from Lemma~\ref{lem: kernel decay} that the associated
velocity $u_0$ belongs to $L^2(\xR^2)$, and hence to 
$H^{s+1}(\xR^2)$ by Proposition~\ref{P203}. 
%We denote by $u=\nabla^\perp\psi$ the solution given by Theorem \ref{T:classique} with $u(0)=u_0$, where $\psi$ is the stream function, and by $\Phi_t$ the flow of $u$. 
Set
$$
V(t,x) \defn (D\Phi_t(x))^{-1}\Phi_t(x)-x. 
$$

\begin{lemma}\label{L:V} 
There holds $\cn V= 0$ and 
\be\label{eq:V} 
\partial_t V=\nabla^\perp\big(\tilde\psi(t,\cdot)\circ\Phi_t\big) \quad \text{where}\quad
\tilde\psi(t, y)\defn2\psi(t, y)-y\cdot\nabla\psi(t,y). 
\ee
Moreover, for all $t\ge 0$, 
\be\label{nV} 
\int_{\xR^2} V(t,x) \cdot u_0(x) \dx= 2t \lA u_0 \rA_{L^2}^2.
\ee
\end{lemma}
\begin{remark}
 For a radial vortex, all the quantities 
of Lemma~\ref{L:V} are explicit. For instance, for
the steady vortex $\psi_0=-\frac18(1-|x|^2)^{4}$ for $|x|<1$ (and
$\psi_0=0$ outside), whose vorticity $\omega_0=2(1-|x|^2)^2(1-4|x|^2)$
is compactly supported with zero 
integral, one finds $V(t,x)=6t\,|x|^2(1-|x|^2)^2Jx$, which is 
tangent to the circles centered at the origin (so that $\cn V=0$) and grows linearly in $t$.
\end{remark}
\begin{proof}
Since $\cn X=2$ is constant, $\cn V=0$ and~\e{eq:V} follow from Lemma~\ref{L71} applied with $X(x)=x$. It remains only to prove \e{nV}.

\smallbreak
\noindent\emph{Step 1: decay at infinity.}
To justify the integral identities by integrating on 
the disk $B_R=\{\la y\ra<R\}$ and letting 
$R$ tend to $+\infty$, we need decay estimates for $\psi$, $u$ and $V$ 
at infinity. Fix $T>0$ and consider $t\in[0,T]$. Recall  that 
$\det D\Phi_t=1$, and that $V(0)=0$ since $\Phi_0= \Id$. 

Since $u$ and $Du$ are bounded on $[0,T]\times\xR^2$, integrating
$\partial_t\Phi_t=u(t,\Phi_t)$ in time and
applying Gr\"onwall's lemma to $\partial_t D\Phi_t=(Du\circ\Phi_t)D\Phi_t$,
we get, for all $ t\in[0,T]$ and $x\in \xR^{2}$, 
$$
\la\Phi_t(x)-x\ra+\la D\Phi_t(x)\ra\le C_T.
$$
Since $\det D\Phi_t=1$, it follows from~\e{n105} that $\la(D\Phi_t)^{-1}\ra=\la D\Phi_t\ra$.

Recall that the vorticity $\omega=\Delta\psi$ 
satisfies $\omega(t,\Phi_t(x))=\omega_0(x)$. Thus $\omega(t)$ is supported in $\Phi_t(\supp\omega_0)$, a set 
contained in a fixed ball for $t\in [0,T]$. Using 
again that $\det D\Phi_t= 1$, we also have
$$
\int_{\xR^2}\omega(t,y)\dy=\int_{\xR^2}\omega_0(x) \dx=0. 
$$

We normalize $\psi$ by
$\psi(t,\cdot)=\frac{1}{2\pi}\log\la\cdot\ra*\omega(t, \cdot)$,
which satisfies $\psi(t,y)=O(\la y\ra^{-1})$ (recall that 
$\omega(t)$ is compactly supported with zero integral).
Indeed, for $y\neq 0$, there holds
$$
\psi(t,y)=\frac{1}{2\pi}\int_{\xR^2}
\big(\log\la y-z\ra-\log\la y\ra\big)\,\omega(t,z)\dz. 
$$
For $\la y\ra\ge 2\la z\ra$, the term in parentheses is $O(\la z\ra/\la y\ra)$, 
and $z$ ranges over a fixed ball containing 
the support of $\omega(t)$. Then Lemma~\ref{lem: kernel decay} applied 
with $y_c=0$ gives $u(t, y)=O(\la y\ra^{-2})$ and
$Du(t, y)=O(\la y\ra^{-3})$, uniformly in $t\in [0,T]$.

To deduce the decay of $V$, we note that the identity $\nabla(y\cdot\nabla\psi)=\nabla\psi+ (D^2\psi)y$ implies that
$$
\tilde\psi = 2\psi-y\cdot\nabla\psi=O(\la y\ra^{-1}),\qquad 
\nabla\tilde\psi=\nabla\psi-(D^2\psi)y=O(\la y\ra^{-2}),
$$
where we used $\la D^2\psi\ra=\la Du\ra$. Using the bound on 
$(D\Phi_t)^{-1}$ and $\la\Phi_t(x)\ra\ge \la x\ra-C_T$, we deduce from \e{eq:V}
and Lemma~\ref{L73} that
$$
\partial_tV=(D\Phi_t)^{-1}(\nabla^\perp\tilde\psi\circ\Phi_t)= O_T(\la x\ra^{-2}). 
$$
Integrating in time from $V(0)= 0$, we find that $V=O_T(\la x\ra^{-2})$.  

Together with $u_0=O(\la x\ra^{-2})$, these bounds imply that the functions $V(t)\cdot u_0$ and $\partial_tV(t)\cdot u_0$ 
are bounded, uniformly in $t\in [0,T]$, by an integrable function. This justifies that 
$$
\fract\int_{\xR^2}V(t,x)\cdot u_0(x) \dx 
=\int_{\xR^2}\partial_tV(t, x)\cdot u_0(x)\dx. 
$$

\smallbreak
\noindent\emph{Step 2: the pairing identity.}  
Set $\psi_0=\psi(0,\cdot)$, so that $u_0=\nabla^\perp\psi_0$ and $\Delta\psi_0=\omega_0$.
Then \e{eq:V} implies that
$$
\int_{\xR^2}\partial_tV\cdot u_0\dx
=\int_{\xR^2}\nabla^\perp(\tilde\psi\circ\Phi_t)\cdot\nabla^\perp\psi_0 \dx 
=\int_{\xR^2}\nabla(\tilde\psi\circ\Phi_t)\cdot\nabla\psi_0\dx. 
$$
Integrating by parts on $B_R$ and denoting by $ n$ the outward unit normal to $\partial B_R$, we have 
$$
\int_{B_R}\nabla(\tilde\psi\circ\Phi_t)\cdot\nabla\psi_0\dx
=-\int_{B_R}(\tilde\psi\circ\Phi_t)\,\omega_0\dx 
+\int_{\partial B_R}(\tilde\psi\circ\Phi_t)\,\partial_n\psi_0\dsigma.
$$
The estimates $\tilde\psi\circ\Phi_t = O_T(\la x\ra^{-1})$
and  $\partial_n\psi_0=O(\la x\ra^{-2})$ 
(see Step~1) imply that the boundary term is $O_T(R^{-2})$ and tends 
to zero as $R\to+\infty$. For $R$
large enough, the support of $\omega_0$ is contained 
in $B_R$, and one can replace $B_R$ by $\xR^2$ in the integral of
$(\tilde\psi\circ\Phi_t)\,\omega_0$. Changing 
variables $y=\Phi_t(x)$ and using $\det D\Phi_t=1$ and $\omega_0(x)=\omega(t,\Phi_t(x))$, we obtain 
$$
\int_{\xR^{2}}(\tilde\psi\circ\Phi_t)\,\omega_0\dx
=\int_{\xR^2}\tilde\psi(\Phi_t(x))\,\omega(t, \Phi_t(x)) \dx 
=\int_{\xR^2}\tilde\psi(y)\,\omega(t, y)\dy. 
$$
Writing $\omega=\Delta\psi$ and $\tilde\psi=2\psi-y\cdot\nabla\psi$, we conclude that 
\be\label{n398}
\fract\int_{\xR^2}V(t,x)\cdot u_0(x) \dx
=-2\int_{\xR^2}\psi\,\Delta\psi\dy +\int_{\xR^2}(y\cdot\nabla\psi)\,\Delta\psi\dy.
\ee
%where $\psi=\psi(t,y)$.

\smallbreak
\noindent\emph{Step 3: Rellich's identity.}
It remains to compute the two integrals on the right-hand side of \e{n398}. 
We claim that
\be\label{n422}
-\int_{\xR^2}\psi\,\Delta\psi\dy=\int_{\xR^2}\la\nabla\psi\ra^2\dy 
\quad\text{and}\quad 
\int_{\xR^2}(y\cdot\nabla\psi)\,\Delta\psi\dy=0.
\ee
The first identity is follows from Green's formula. Indeed,  
$$
-\int_{B_R}\psi\,\Delta\psi\dy =\int_{B_R}\la \nabla\psi\ra^2\dy 
-\int_{\partial B_R}\psi\,\partial_n\psi\dsigma.
$$
On $\partial B_R$, we have
$\psi=O(R^{-1})$ and $\partial_n\psi=O(R^{-2})$, so that the boundary term 
is $O(R^{-2})$. Letting $R$ go to 
$+\infty$, we obtain the first identity in~\e{n422}. 

The second identity in \e{n422} is specific to dimension~$2$. Indeed, on $\xR^d$,
the computation below, where $\cn y=d$, gives $\frac{d-2}{2}\int\la \nabla\psi\ra^2\dy$,
up to boundary terms. To prove it, we use Rellich's 
multiplier $y\cdot\nabla\psi$ and integrate by parts on
$B_R$. Using the identity $\nabla(y\cdot\nabla\psi)\cdot\nabla\psi=\la \nabla\psi\ra^2+\frac{1}{2}\,y\cdot\nabla(\la\nabla\psi\ra^2)$, 
we obtain
$$
\int_{B_R}(y\cdot\nabla\psi)\,\Delta\psi\dy
=\int_{\partial B_R}(y\cdot\nabla\psi)\,\partial_n\psi\dsigma 
-\int_{B_R}\Big(\la\nabla\psi\ra^2+\frac{1}{2}\,y\cdot\nabla\big(\la \nabla\psi\ra^{2}\big)\Big)\dy. 
$$
A second integration by parts, where we use $\cn y=2$, gives
$$
\frac{1}{2}\int_{B_R}y\cdot\nabla\big(\la\nabla\psi\ra^2 \big)\dy
=\frac12\int_{\partial B_R}(y\cdot n)\la\nabla\psi\ra^2\dsigma-\int_{B_R}\la\nabla\psi\ra^2\dy. 
$$
The integrals of $\la\nabla\psi\ra^{2}$ over $B_R$ cancel, and we are left with 
$$
\int_{B_R}(y\cdot\nabla\psi)\,\Delta\psi\dy
=\int_{\partial B_R}(y\cdot\nabla\psi)\,\partial_n\psi\dsigma
-\frac{1}{2}\int_{\partial B_R}(y\cdot n)\la\nabla\psi\ra^2\dsigma . 
$$
On $\partial B_R$, we have $y=Rn$, so that $y\cdot\nabla\psi=R\,\partial_n\psi$ 
and $y\cdot n=R$. Substituting, we
obtain Rellich's identity~\cite{Rellich1940} (for semilinear elliptic equations, this identity is known as
the Pohozaev identity \cite{Pohozaev1965})
$$
\int_{B_R}(y\cdot\nabla\psi)\,\Delta\psi\dy
=R\int_{\partial B_R}
\Big((\partial_n\psi)^{2}-\frac{1}{2}\la \nabla\psi\ra^2\Big)\dsigma.
$$
The right-hand 
side is $O(R^{-2})$, since $\nabla\psi=O(R^{-2})$ on $\partial B_R$.
Letting $R$ go to $+\infty$, we obtain the second identity in~\e{n422}.

We have thus proved that 
$$
\fract\int_{\xR^2}V(t,x)\cdot u_0(x)\dx 
=2\int_{\xR^2}\la\nabla\psi(t,y)\ra^2\dy =2\lA u_0\rA_{L^2}^{2}, 
$$
where we used $\la\nabla\psi\ra=\la u\ra$ and the conservation of energy, 
$\lA u(t)\rA_{L^2}=\lA u_0\rA_{L^2}$. Integrating in time from  $0$ to $t$ and using $V(0)=0$, we 
obtain~\e{nV}. 
\end{proof}

\subsection{Unbounded deformation of the flow}

\begin{proposition}\label{P129}
Let $s>1$ and consider a non-zero initial vorticity $\omega_0\in H^s(\xR^2)$ 
with compact support and zero integral. Then the associated flow satisfies
$$
\sup_{t\ge 0}\lA D\Phi_t \rA_{L^\infty}=+\infty.
$$
\end{proposition}

\begin{proof}
Assume, by contradiction, that 
$$
\sup_{t \ge 0}\lA D \Phi_t \rA_{L^\infty}= A < +\infty.
$$
Since $\det(D\Phi_t)= 1$, the inverse matrix satisfies
$\la (D\Phi_t)^{-1}\ra= \la D\Phi_t \ra\le A$, 
which implies that
\be\label{n220}
\frac{1}{A}\la x-y\ra\le \la \Phi_\tau(x)-\Phi_\tau(y)\ra \le A \la x- y\ra 
\quad \text{for all } x,y \in \xR^2 \text{ and }\tau \ge 0.
\ee
Notice that $A\ge 1$ by \e{n106}. Denote by $K_0$ the support
of $\omega_0$. Without loss of generality, one can assume that $0 \in K_0$. Fix a radius 
$R\ge 2 A^2 \diam(K_0)$ (to be chosen 
at the end of the proof) and set
$\Omega \defn \{ \la x\ra\ge R\}$ and $B_R\defn \{\la x\ra<R\}$. 
We split the integral in \e{nV} into
the contributions of $\Omega$ and $B_R$. We will prove that, for all $t\ge 0$, 
\be\label{n221}
\begin{aligned}
&\la \int_{\Omega} V(t, x) \cdot u_0(x) \dx\ra\le \frac{\Gamma}{R^2} t, 
\\
&\la \int_{B_R} V(t, x) \cdot u_0(x) \dx\ra\le \Gamma_R\big(1+\sqrt{t}\big),
\end{aligned}
\ee
where $\Gamma$ depends only on $A$ and
$\omega_0$, and $\Gamma_R$ depends only on $A$, $\omega_0$ and $R$. Since $\Gamma$ does not depend 
on $R$, and since the left-hand side of~\e{nV} is equal to
$2\lA u_0\rA_{L^2}^2\,t$, these two estimates are incompatible for
$R$ and $t$ large enough.

\smallbreak
\noindent\emph{Step 1: decay of the displacement far from the origin.}
At any time $\tau\ge 0$, the support of the vorticity is 
$K_\tau=\Phi_\tau(K_0)$. Set $y_c(\tau)\defn \Phi_\tau(0) \in K_\tau$. It follows from \e{n220}
that $\diam(K_\tau)\le A \diam(K_0)$, and hence $\la y - y_c(\tau)\ra\le A \diam(K_0)$
for all $y \in K_\tau$. Consequently,
the first moment of the vorticity satisfies 
\begin{align*}
M_1(\tau) &\defn \int_{\xR^2}\la y-y_c(\tau)\ra \la \omega(\tau, y) \ra \dy \\
&\le A \diam(K_0) \int_{K_\tau}\la \omega(\tau, y) \ra \dy=
A \diam(K_0) \lA \omega_0 \rA_{L^1}\eqqcolon \mathcal{M}, 
\end{align*}
where we used again the fact that $\Phi_\tau$ 
preserves area. Consider $x\in \Omega$, for which \e{n220} implies 
that, for all $\tau\ge 0$,
$$
\la \Phi_\tau(x) - y_c(\tau) \ra=\la \Phi_\tau(x) -\Phi_\tau(0)\ra\ge \frac{\la x \ra}{A}.
$$
Combined with $\la x\ra\ge R$ and $\diam(K_\tau) \le A \diam(K_0)$, this implies that 
$$
\la \Phi_\tau(x) -y_c(\tau)\ra \ge \frac{R}{A}\ge 2A \diam(K_0) \ge 2\sup_{y\in K_\tau}\la y-y_c(\tau)\ra.  
$$
One can therefore apply Lemma~\ref{lem: kernel decay} with $y_c=y_c(\tau)$ to obtain  
$$
\la u(\tau, \Phi_\tau(x)) \ra\le \frac{C M_1(\tau)}{\la \Phi_\tau(x)- y_c(\tau)\ra^2}
\le \frac{C \mathcal{M} }{(\la x \ra / A)^2}=\frac{C A^2 \mathcal{M}}{\la x\ra^2},
$$
and 
$$
\la Du(\tau, \Phi_\tau(x))\ra\le \frac{C M_1(\tau)}{\la\Phi_\tau(x)-y_c(\tau)\ra^3}
\le \frac{C \mathcal{M} }{(\la x\ra /A)^3}=\frac{C A^3\mathcal{M}}{\la x\ra^3}.
$$
We now introduce the displacement $\xi(t,x)\defn\Phi_t(x)-x$, which
satisfies $\partial_t \xi(t,x)= u(t,\Phi_t(x))$ 
and $\partial_t D\xi(t, x)= (Du\circ \Phi_t)D\Phi_t$. Integrating in time, we
obtain, for all $t\ge 0$ and $x\in\Omega$, 
$$
\la \xi(t,x) \ra\le \int_0^t \la u(\tau, \Phi_\tau(x))\ra\dtau\le \frac{C A^2\mathcal{M} t}{\la x\ra^2},  
$$
$$
\la D\xi(t, x)\ra \le \int_0^t \la Du(\tau, \Phi_\tau(x))\ra \la D\Phi_\tau(x)\ra\dtau\le \frac{C A^4 \mathcal{M} t}{\la x\ra^3}.
$$

\smallbreak
\noindent\emph{Step 2: the contribution of $\Omega$.}
We write $V $ in terms of the displacement:
\be\label{n222}
V(t,x)= (D\Phi_t(x))^{-1}\big( \Phi_t(x)-D\Phi_t(x) x \big)= (D\Phi_t(x))^{-1}\big( \xi(t,x) - D\xi(t, x) x \big).
\ee
It then follows from Step~1 and the bound $\la (D\Phi_t)^{-1}\ra \le A$ that, for all $x\in \Omega$, 
\begin{align*}
\la V(t,x)\ra &\le A \big( \la \xi(t,x) \ra+ \la x \ra \la D\xi(t, x)\ra \big)\\ 
&\le A \left( \frac{C A^2 \mathcal{M} t}{\la x
\ra^2}+\la x\ra \frac{C A^4 \mathcal{M} t}{\la x\ra^3}\right)\le \frac{2 C A^5 \mathcal{M} t}{\la x\ra^2}.  
\end{align*}

On the other hand, Lemma~\ref{lem: kernel decay}, applied at time $t=0$, implies that
$$
\la u_0(x)\ra\le \frac{C \mathcal{M}}{\la x \ra^{2}}\qquad\text{for all }x\in\Omega.  
$$
Therefore,
\begin{align*}
\la \int_{\Omega} V(t,x) \cdot u_0(x) \dx\ra &\le \int_{\Omega}\frac{2 C A^5 \mathcal{M} t}{\la x\ra^2}\frac{C \mathcal{M}}{\la x\ra^2}\dx\\ 
&= 2\pi (2 C^2 A^5 \mathcal{M}^2) t \int_R^\infty
\frac{\dr}{r^3}=\frac{2\pi C^{2} A^5 \mathcal{M}^2}{R^2} t,
\end{align*}
which is the first estimate in~\e{n221}, with $\Gamma=2\pi C^{2} A^5 \mathcal{M}^2$.

\smallbreak
\noindent\emph{Step 3: the contribution of $B_R$.}  
We begin by estimating the displacement $\xi(t,0)= y_c(t)$ of the origin. Since 
$\partial_t\xi=u(t, \Phi_t)$ and $\Phi_t$ 
preserves area, the conservation of energy implies that 
$$
\lA \xi(t,\cdot) \rA_{L^2} \le t\lA u_0 \rA_{L^2}.  
$$
On the other hand, \e{n220} implies that 
$$
\la \xi(t,x)\ra \ge \la \xi(t, 0)\ra-(A+1)\la x \ra.
$$
In particular, $\la\xi(t, x)\ra\ge \frac{1}{2}\la\xi(t,0)\ra$ on the ball of radius $\rho=\frac{\la \xi(t, 0) \ra}{2(A+1)}$, and hence 
$$
\lA \xi(t,\cdot) \rA_{L^2}^2\ge \frac{\pi \la \xi(t,0)\ra^4}{16(A+ 1)^{2}}. 
$$
By comparing these two estimates, we get 
$$
\la \xi(t,0)\ra\le 2 \sqrt{\frac{A + 1}{\sqrt{\pi}}\lA u_0 \rA_{L^2}}\sqrt{t}\eqqcolon C_A \sqrt{t}.
$$

Consequently, for all $x \in B_R$,
$$
\la \xi(t,x)\ra\le \la \xi(t,0) \ra+(A+1)\la x \ra \le C_A \sqrt{t}+(A + 1)R, 
\qquad 
\la D\xi(t,x)\ra\le A+1,
$$
where we used $D\xi=D\Phi_t-I_2$ for the second estimate. We then deduce from \e{n222} that, for all $x\in B_R$,
$$
\la V(t, x)\ra\le A \big( C_A \sqrt{t}+(A+1)R +R(A+1) \big)= A C_A \sqrt{t}+2 A(A+1) R.  
$$
Multiplying by the locally integrable function $\la u_0\ra$ and integrating over $B_R$, we obtain 
$$
\la \int_{B_R}V(t, x)\cdot u_0(x) \dx\ra\le \left( A C_A \sqrt{t}+2 A(A+1) R \right) \int_{B_R}\la u_0(x) \ra\dx,
$$
which implies the second estimate in~\e{n221}. 

\smallbreak
\noindent\emph{Conclusion.} 
It follows from \e{nV} and \e{n221} that, for all $t\ge0$,
$$
2\lA u_0 \rA_{L^2}^{2} t\le \Gamma_R\big(1+\sqrt{t}\big)+\frac{\Gamma}{R^2}\, t. 
$$
One can choose the radius $R $, independently of $t$ (the constant $A$ does not depend on $t$), such that $R\ge 2 A^2 \diam(K_0)$ and 
$$
\frac{\Gamma}{R^2}\le \lA u_0 \rA_{L^2}^2.
$$
With this choice of $R$, we have
$$
\lA u_0 \rA_{L^2}^2 t \le \Gamma_R\left(1+\sqrt{t}\right).
$$
As $\omega_0\neq0$, we have $\lA u_0\rA_{L^2}>0$, and this inequality fails for $t$
large enough, which is the desired contradiction. 
\end{proof}

\section{Unbounded deformation on \texorpdfstring{$\xT^2$}{T2}}\label{S473}

In this section we prove 
Theorem~\ref{T104} and deduce from it, together with a logarithmic bound
on the polar angle, that the differential of the flow map is unbounded in time for all non-zero initial data on $\xT^2$. 

\subsection{Setting and statements}\label{S81}

With the notations of Section~\ref{S:notation}, let $s >1$ and consider a
non-zero initial vorticity $\omega_0\in H^s_0(\xT^2)$, the solution $u=\nabla^\perp\psi$ given by Theorem \ref{T:classique},
and the corresponding flow $\Phi_t$, lifted to $\xR^2$ as in Section \ref{S:classique}. Recall that, for $k\in\xZ^2$,
\be\label{n223} 
\Phi_t(x+2\pi k)=\Phi_t(x)+2\pi k,\qquad \det D\Phi_t=1,\qquad
\omega\left(t,\Phi_t(x)\right)=\omega_0(x).
\ee
In particular, the enstrophy $\av{\omega(t)^2}=\av{\omega_0^{2}}$ is conserved, and is strictly positive since $\omega_0\neq0$.

We set
\be\label{n224}
B_t(x)=Du\big(t,\Phi_t(x)\big), 
\ee
so that $\partial_tD\Phi_t=B_tD\Phi_t$. In view of~\e{n105}, $\Phi_t$
and $\Phi_t^{-1}$ are $\lA D\Phi_t\rA_{L^\infty}$-Lipschitz 
on  $ \xR^2$ (to say that $\lA D\Phi_t\rA_{L^\infty}$ stays bounded is to say 
that the fluid does not filament). 

Writing $k(N)=N_{21}-N_{12}$ for a $2\times2$ matrix $N$, we deduce from $k(Du)=\omega$ and from \e{n223} that 
\be\label{n225}  
k(B_t)=\omega_0 .
\ee

The following result is Theorem \ref{T103} for $M= \xT^2$.

\begin{theorem}\label{T105}
Let $s>1$ and let $\omega_0\in H^s_0(\xT^2)$ be non-zero. Then 
$$
\sup_{t \ge 0}\lA D\Phi_t\rA_{L^\infty}=+\infty . 
$$
\end{theorem}

The rest of this section is devoted to the proof of Theorem \ref{T105}.

\subsection{The polar angle}\label{S82}

Recall from Section~\ref{S471} that
$$
D\Phi_t=R_{\theta_t}S_t \quad\text{where}\quad S_t=\left((D\Phi_t)^\top D\Phi_t\right)^{1/2}.
$$
To check the regularity and the periodicity of $\theta_t$, 
observe that the matrix $D\Phi_t$ is periodic
in $x$ and $C^1$ in $t$, and that the square root is analytic on positive 
definite matrices. Hence $S_t$ is periodic
in $x$ and $C^1$ in $t$, as is $R_{\theta_t}=D\Phi_tS_t^{-1}$. Consequently, $\theta_t(x)$ and $\partial_t\theta_t(x)$ 
are continuous in $(t,x)$. The angle $\theta_t$ is also
periodic, since $\theta_t(x+2\pi k)$ and $\theta_t(x)$ are
continuous in $t$, equal modulo $2\pi$, and coincide at $t=0$. 

The next lemma shows that,  up to an error smaller than $\pi/2$, $\theta_t(x)$ is the angle by
which every tangent vector based at $x$ has turned between the times $0$ and $t$. 

\begin{lemma}\label{L122}
Let $v$ be a unit vector, and denote by $\beta_v(t,x)\in(-\pi/2,\pi/2)$
the oriented angle from $v$ to $S_t(x)v$. Then $\theta_t+\beta_v$ is the oriented angle from
$ v$ to $D\Phi_t(x)v$, chosen continuously in time with value $0$ at $ t = 0$.
\end{lemma}

\begin{proof}
The angle $\beta_v$ is well defined in 
$(-\pi/2,\pi/2)$ and continuous in $t$, as $S_tv\cdot v>0$, and since $D\Phi_tv=R_{\theta_t}(S_tv)$, the angle from 
$v$ to $D\Phi_tv$ is $\theta_t+\beta_v$ modulo
$2\pi$. At $t=0$, we have $\theta_0=0$, and 
$\beta_v=0$ since $S_0=I$. Since both angles are continuous in $t$, equal modulo $2\pi$, and
coincide at $t=0$, they are equal.
\end{proof}

Recall that $\mathscr A(t)=\av{\omega_0\,\theta_t}$ (see~\e{n395}). Writing $\omega_0=\omega_0^+-\omega_0^-$,
where $\av{\omega_0^+}=\av{\omega_0^-}$ since $\av{\omega_0}=0$, we have
$$
\mathscr A(t)=\av{\omega_0^+\theta_t}-\av{\omega_0^-\theta_t}.
$$
Thus $\mathscr A$ compares the average rotation angle of the regions of positive vorticity with that 
of the regions of negative vorticity. In particular, a uniform rotation of the whole fluid does not contribute to $\mathscr A$.

\subsection{The production identity}\label{S83}

Let us now prove Theorem \ref{T104}. For concision we write $D\Phi_t= R_\theta S$ 
with $\theta=\theta_t$ and $S=S_t$, 
we denote by a dot the derivative $\partial_t$, and we set 
$$
Q=\dot SS^{-1},\qquad H=S^{-1/2}\dot SS^{-1/2}.
$$
The trace of $Q$ vanishes, since 
$\tr Q=\tr(S^{-1}\dot S)= \frac{\diff}{\dt}\log\det S$
and $\det S=1$. The matrix $H$ is symmetric, with $\tr H=0$ and $\det H=\det\dot S$. 

\begin{lemma}\label{L123} 
For all $t\ge 0$ and all $x\in \xR^2$,
\be\label{n233}
R_\theta^\top BR_\theta=\dot\theta J+Q, \qquad \omega_0=2\dot\theta+k(Q), 
\ee
and
\be\label{n234}  
\omega_0\,\dot\theta=\dot\theta^{\,2}+\frac12\la H\ra_{\Fro}^2+\det B .
\ee
\end{lemma}

\begin{proof}
Since $J$ commutes with the rotations, $\partial_tR_\theta= \dot\theta R_\theta J$. 
Differentiating $D\Phi_t=R_\theta S$ and using 
$\partial_tD\Phi_t=BD\Phi_t$, we get 
$\dot\theta R_\theta JS+R_\theta\dot S=BR_\theta S$. Multiplying by $R_\theta^\top$ on 
the left and by $S^{-1}$ on the right, 
we obtain the first identity. The second one follows by applying $k$ to the first
identity and using \e{n225}, because the function $k(N)=\tr(NJ^\top)$ is invariant
under conjugation by rotations and satisfies $k(J)=2$. 

To prove \e{n234}, recall that, for a trace-free matrix $Q$,
$\det(\dot\theta J+Q)=\dot\theta^{\,2}+\dot\theta\,k(Q)+\det Q$. 
Substituting $k(Q)=\omega_0-2\dot\theta$ into the identity
$\det B=\dot\theta^{\,2}+\dot\theta\,k(Q)+\det\dot S$ (which follows from 
the first identity in~\e{n233} and from $\det Q=\det\dot S$), we get
$\omega_0\dot\theta= \dot\theta^{\,2}-\det\dot S+ \det B$. Finally, since
$H$ is symmetric and trace-free, the eigenvalues of $H$ are
$\pm h$, and hence $-\det\dot S=-\det H=h^2=\frac{1}{2}\la H\ra_{\Fro}^2$.
This proves \e{n234}.
\end{proof}

The first identity in \e{n233} means that, in the frame of the
fluid element, the velocity gradient is the sum
of a rotation at the rate $\dot\theta$ and of the stretching 
rate $Q$. The second one is the classical relation between vorticity 
and rotation, except that in general the element does not rotate 
at the rate $\omega_0/2$ (the difference $k(Q)$ comes from the rotation of the axes of the deformation).

\begin{proof}[Proof of Theorem \ref{T104}]
Differentiating under the mean (the functions 
$\theta$ and $\dot\theta$ are continuous in $(t,x)$ and periodic in $x$), we find that 
$\mathscr A'(t)=\av{\omega_0\dot\theta}$. Moreover, $\mathscr A(0)=0$ since $\theta_0=0$. In two dimensions,
$$
 \det Du = \partial_1 u_1 \partial_2 u_2-\partial_2 u_1 \partial_1 u_2 =\partial_1(u_1\partial_2u_2)-\partial_2(u_1\partial_1u_2) 
$$
hence $\av{\det Du}=0$ by periodicity
(for $u(t)\in C^1$, this follows by approximation). As the change of variables 
$y=\Phi_t(x)$ preserves area, we deduce that $\av{\det B_t}=\av{\det Du(t)}=0$. 
Averaging~\e{n234} then yields~\e{n396}. 

To prove~\e{n397}, fix $(t, x)$. Since $k$ and $\la\cdot\ra_{\Fro}$ are invariant
under conjugation by rotations, one can work in a direct orthonormal 
basis of eigenvectors of $S$, so that $S= \mathrm{diag}(r, r^{-1})$ with $1\le r\le \lA D\Phi_t\rA_{L^\infty}$. 
Writing $\dot S=\begin{pmatrix}a&b\\ b&d\end{pmatrix}$, we have
$$
k(Q)=\frac br-br=-(r-r^{-1})b, \qquad
\frac{1}{2}\la H\ra_{\Fro}^2= \frac12\Big(\frac{a^2}{r^2}+2b^2+d^{2}r^2\Big)\ge b^{2} .
$$
Using the second identity in \e{n233} and the Cauchy--Schwarz inequality, we obtain 
\begin{align*}
\omega_0^2&=\big(2\dot\theta-(r-r^{-1})b\big)^2\le \big(4+(r-r^{-1})^2\big)\big(\dot\theta^{\,2}+b^{2}\big)\\
&=(r + r^{-1})^2\big(\dot\theta^{\,2}+b^{2}\big)\le 4\lA D\Phi_t\rA_{L^\infty}^2\Big(\dot\theta^{\,2}+\frac12\la H\ra_{\Fro}^2\Big). 
\end{align*}
Averaging the previous inequality and using \e{n396}, we obtain \e{n397}.
\end{proof}

For a round element ($r=1$), the second identity 
in \e{n233} forces $\dot\theta=\omega_0/2$, that 
is, the element turns at half the rate of its vorticity, while for 
a filament ($ r\gg1$), the vorticity can be carried by the term $(r-r^{-1})b$, without 
rotation of the element. 
\begin{remark}
The Euler equation enters only through the transport 
of the vorticity, $k(B_t)=\omega_0$, and through 
the cancellation $\av{\det Du}=0$. The rest of the argument is kinematic. 
\end{remark}
\begin{remark}[The shear flow]\label{R107}
For $\omega_0=\sin x_1$, the solution is 
stationary, with $u=(0, -\cos x_1)$ and $\Phi_t(x)=(x_1,x_2-t\cos x_1)$. Then $\tr D\Phi_t=2$ 
and $k(D\Phi_t)=t\omega_0$. Using that $\tr(R_\theta S)=\cos\theta\,\tr S$ 
and $k(R_\theta S)=\sin\theta\,\tr S$ for all
symmetric matrices $S$, we deduce that $\theta_t=\arctan(t\omega_0/2)$. 
Hence, as $t\to+\infty$, $\theta_t\to\pm\pi/2$ where $\omega_0\neq0$, and $\mathscr A(t)\nearrow\frac\pi2\av{\la \omega_0\ra}$. 
In particular, $\mathscr A$ is 
non-decreasing and bounded, while $\lA D\Phi_t\rA_{L^\infty}\sim t$. 
More precisely, for $t>0$, 
$$
\mathscr A'(t)=\frac{2}{t^2}\left(1-\frac{2}{\sqrt{4+t^2}}\right) 
\sim\frac{2}{t^2},
\qquad 
\frac{\av{\omega_0^2}}{4\lA D\Phi_t\rA_{L^\infty}^2}\sim\frac{1}{8t^2}. 
$$
Thus the production is integrable and has the same order as the lower 
bound \e{n397}, so the power $\lA D\Phi_t\rA_{L^\infty}^{-2}$
cannot be improved. Although the vorticity 
is stationary, $\mathscr A$ grows, since it depends on the Lagrangian flow and not only on the vorticity. 
\end{remark}

\subsection{A logarithmic bound on the angle}\label{S84}  

To bound $\mathscr A(t)$ from above, assuming that $\lA D\Phi_\tau\rA_{L^\infty}\le A$
for $0\le\tau\le t$, it suffices, in view of Lemma \ref{L122} and since
$\la\beta_{e_1}\ra<\pi/2$, to bound $\theta_t+\beta_{e_1}$.

To introduce the chord, fix $t\ge 0$, a point
$x\in\xR^2$ and $j\in \{1,2\}$. For $0<\rho \le 2\pi$, set
\be\label{n235}
c(\rho)=\Phi_t(x+\rho e_j)-\Phi_t(x). 
\ee
It never vanishes, since $\Phi_t$ is injective, and
$ \partial_\rho c(\rho)=D\Phi_t(x+\rho e_j)e_j$. Moreover,
\be\label{n450} 
c(2\pi)=2\pi\,e_j, 
\qquad
\frac{c(\rho)}{\rho}\xrightarrow[\rho\to0]{}D\Phi_t(x)e_j . 
\ee
The first identity follows from~\e{n223}, and the second one from the definition of $D\Phi_t(x)$.  

We next prove two bounds on the rate $\la c\wedge\partial_\rho c\ra/\la c\ra^2$
at which the chord rotates when $\rho$
increases. Fix $0<\gamma<\min(1,s-1)$, and 
denote by $\Lambda_t\defn [D\Phi_t]_{C^\gamma}$ the H\"older seminorm of
$D\Phi_t$ (it is finite because $D\Phi_t-I\in H^s\subset C^{0,\gamma}$, 
see \e{n204} in the appendix). Assume that $\lA D\Phi_t\rA_{L^\infty}\le A$.
Then $\Phi_t$ and its inverse are $A$-Lipschitz, so that 
$\rho/A\le \la c(\rho)\ra\le A\rho$ and $\la\partial_\rho c(\rho)\ra\le A$. Consequently,
\be\label{n451} 
\frac{\la c\wedge\partial_\rho c\ra}{\la c\ra^2}\le \frac{\la\partial_\rho c\ra}{\la c\ra} \le \frac{A^2}{\rho} .
\ee
This bound is not integrable at $\rho= 0$, and for small $\rho$ we use instead that the chord is almost linear in $\rho$. Setting
$d(\rho)=c(\rho)-\rho\,\partial_\rho c(\rho)$, we have 
$$
d(\rho)=-\int_0^\rho\big[D\Phi_t(x+\rho e_j)-D\Phi_t(x+q e_j)\big]e_j\,\diff q,
$$
so that $\la d(\rho)\ra\le \Lambda_t\rho^{1+\gamma}/(1+\gamma)$. Using the identity $c\wedge\partial_\rho c=-\rho^{-1}c\wedge d$, we deduce that
\be\label{n452}
\frac{\la c\wedge\partial_\rho c\ra}{\la c\ra^2}
\le\frac{\la d\ra}{\rho\la c\ra}
\le\frac{A\Lambda_t}{1+\gamma}\,\rho^{\gamma-1}.
\ee

\begin{lemma}[Reconstruction of the angle]\label{L124}
For all $t\ge 0$ and all $x\in\xR^2$, the following integral converges and
\be\label{n470}
\theta_t(x)+\beta_{e_j}(t,x) = -\int_0^{2\pi}\frac{c(\rho)\wedge\partial_\rho c(\rho)}{\la c(\rho)\ra^2}\diff\rho .
\ee
\end{lemma}

\begin{proof}
The convergence at $\rho=0$ follows from~\e{n452}, 
applied with $A=\lA D\Phi_t\rA_{L^\infty}$. The chord $c$ does not vanish, so that one can write 
$c(\rho)=\la c(\rho)\ra(\cos\alpha(\rho), \sin\alpha(\rho))$ with
$\alpha$ continuous on $(0,2\pi]$, and
the integrand is equal to $\alpha'(\rho)$. Integrating in $\rho$, we 
find that the right-hand side is equal to $\alpha(0^+)-\alpha(2\pi)$,
that is, by~\e{n450}, to the angle from $e_j$ to $D\Phi_t(x)e_j$
modulo $2\pi$. By Lemma~\ref{L122}, $\theta_t+\beta_{e_j}$ 
is also equal to this angle modulo $2\pi$. Since 
$\lA D\Phi_t\rA_{L^\infty}$ and $\Lambda_t$ are
bounded on every bounded time interval, the right-hand side is continuous in $t$ by~\e{n452} and
dominated convergence, so that the two sides are continuous in
$t$, equal modulo $2\pi$, and vanish at $t=0$ (since then $c(\rho)=\rho e_j$). 
They are therefore equal. 
\end{proof}

\begin{lemma}\label{L401} 
Assume that $\sup_{0\le \tau\le t}\lA D\Phi_\tau\rA_{L^\infty} \le A$.
Then there exists $C_1$, depending only on $\omega_0$, $\gamma$ and $A$, 
such that, for all $x\in \xR^2$, $\la\theta_t(x)\ra \le C_1(1+\log(1+t))$.
\end{lemma}

\begin{proof}
\smallbreak
\noindent\emph{Step 1: growth of $\Lambda_t$.} Since $\omega_0\in C^{0,\gamma}$ and since the vorticity is transported by an 
$A$-Lipschitz map, $\lA \omega(\tau)\rA_{C^\gamma}$ 
is bounded uniformly in $\tau\in[0,t]$, and the Schauder estimate for the Biot--Savart 
law then implies that $\Lambda_u\defn \sup_\tau[Du(\tau)]_{C^\gamma}$
is finite. Now fix $x,y\in \xR^2$. The difference 
$\delta(\tau)=D\Phi_\tau(x)-D\Phi_\tau(y)$ satisfies $\delta(0)=0$ and 
$$
\partial_\tau\delta=B_\tau(x)\,\delta+\big(B_\tau(x)-B_\tau(y)\big)D\Phi_\tau(y). 
$$
The solutions of the equation $\partial_\tau\delta = B_\tau(x)\delta$
are of the form $D\Phi_\tau(x)D\Phi_\sigma(x)^{-1}\delta(\sigma)$, and
$\la D\Phi_\tau(x)D\Phi_\sigma(x)^{-1}\ra\le A^{2}$. As the last term
is bounded by $\Lambda_uA^{1+\gamma}\la x-y\ra^\gamma$, Duhamel's
formula then implies that
$\la\delta(t)\ra\le \Lambda_uA^{3+\gamma}t\,\la x-y\ra^\gamma$, 
namely, $\Lambda_t\le \Lambda_uA^{3+\gamma}t$.

\smallbreak
\noindent\emph{Step 2: splitting the range of $\rho$.} It follows from~\e{n470}, \e{n451} and \e{n452} that
$$
\la\theta_t(x)+\beta_{e_j}(t,x)\ra\le \int_0^{2\pi}
\min\Big(\frac{A^2}{\rho},\ \frac{A\Lambda_t}{1+\gamma}\rho^{\gamma-1}\Big)\diff\rho
\le\frac{A^2}{\gamma}\Big[1+\log_+\frac{\Lambda_t\,(2\pi)^{\gamma}}{A(1+\gamma)}\Big],
$$
where $\log_+ r=\log(\max\{1,r\})$ for
$r\ge 0$. To obtain the second inequality, 
we split the integral at $\rho_*$ such that  $\rho_*^\gamma=A(1+\gamma)/\Lambda_t$,
using the second bound below $\rho_*$ and the first
one above it. (If $\Lambda_t=0$, the integrand vanishes.) This proves the lemma, since 
$\Lambda_t\le \Lambda_uA^{3+\gamma}t$ by Step~1 and $\la\beta_{e_1}\ra<\pi/2 $. 
\end{proof}

\subsection{Proof of Theorem \ref{T105}}\label{S87} 

\begin{proof}
Assume, by contradiction, that $\lA D\Phi_t\rA_{L^\infty}\le A$ for all  $t\ge 0$. On the one hand,
\e{n397} and $\mathscr A(0)=0$ imply that $\mathscr A(t)\ge \av{\omega_0^2}t/(4A^2)$. 
On the other hand, it follows from Lemma~\ref{L401} and
$\la \mathscr A(t)\ra\le \av{\la \omega_0\ra}\sup\la \theta_t\ra$ that
$\la \mathscr A(t)\ra\le C_2(1+ \log(1+t))$, for some constant $C_2$ depending only on
$\omega_0$, $\gamma$ and $A$. Since $\av{\omega_0^{2}}>0$, these two bounds are incompatible for $t$ large enough. 
\end{proof}

\section{Lagrangian stretching implies instability}\label{S:Koch}
To deduce the $H^s$ results of Remark~\ref{R:Hs} and 
Theorem \ref{T:R2} from Theorem \ref{T105} and Proposition~\ref{P129}, 
we use the following proposition. The $W^{1,p}$ result, Theorem \ref{T11}, is proved in Section~\ref{S507}. 
%Recall that $M$ denotes $\xT^2$ or $\xR^2$, and that $H^s_0(\xR^2)=X^s(\xR^2)$ (see Section~\ref{S:classique}).

\begin{proposition}\label{PK:7.1}
Let $s > 1$. Consider $\omega_0\in X^s(M)$ such that
$\omega_0\in H^k(M)$ for all integers $k\ge 0$,
and such that the inverse flow $\varphi_t^{\omega_0}$ satisfies 
\begin{equation}\label{n:1201}
\sup_{t \ge 0} \lA D\varphi_t^{\omega_0}\rA_{L^\infty(M)} =+\infty.
\end{equation}
Then, for all $r > 0$ and all $N>0$, there exists an initial
vorticity $f_{0} \in \bar{B}_{X^s(M)}(\omega_{0}, r)$ such that 
\begin{equation*}
\sup_{t \ge 0} \lA f_{0} \circ \varphi_t^{f_0}\rA_{H^s(M)} > N.
\end{equation*}
\end{proposition}
\begin{remark}\label{R108}
By \e{n106}, the assumption \e{n:1201} is equivalent to 
$\sup_{t\ge 0}\lA D\Phi_t^{\omega_0}\rA_{L^\infty} = +\infty$.
Hence, it follows from Theorem \ref{T105} that~\e{n:1201} holds for all 
smooth non-zero $\omega_0$ on $\xT^2$, and from 
Proposition~\ref{P129} that it holds for all smooth non-zero $\omega_0$ on $\xR^2$ with 
compact support and zero integral. \end{remark}
\begin{remark}
Proposition \ref{PK:7.1} is a version in
$H^{s}$ of an instability result of Koch~\cite[Theorem~3]{zbMATH01801548}, who proved that the stretching 
of the Lagrangian flow produces the growth of Eulerian norms. Our proof shows that this instability
relies on the stability of the data-to-solution map under high-frequency  perturbations. Notice that Bourgain and Li~\cite[Proposition~4.2]{BourgainLi2015} used also a similar high-frequency perturbation to turn large Lagrangian deformation into norm inflation
in the critical space $\dot{H}^1(\xR^2)$. 
\end{remark}

We will compare Sobolev norms on $M$ and on $\xR^2$ for functions supported in small balls. 

\begin{lemma}\label{L419}
Let $\sigma\in [0,s+ 1]$. There exists $c_M\in (0,1]$ such that, for all balls $B$ of radius at most $1$ and 
all $f\in H^\sigma(\xR^2)$ supported in  $B$,
\be\label{n471}
c_M\lA f\rA_{H^\sigma(\xR^2)}\le \lA f\rA_{H^\sigma(M)}\le c_M^{-1}\lA f\rA_{H^\sigma(\xR^2)}, 
\ee
where, when $M=\xT^2$, $f$ is identified with its periodization.
\end{lemma}
\begin{proof}
When $M=\xR^2$, there is nothing to prove. When
$M=\xT^2$, both norms are invariant under translations, so that we 
may assume that $B$ is centered at $0$. Then $B$ is contained in $B(0,1)$,
which is a fixed compact subset of the open cell $(-\pi,\pi)^2$. Fix
$\chi\in C^\infty_c((-\pi, \pi)^{2})$ with $\chi=1$ on $B(0, 1)$. 
For  $f$ supported in $B(0,1)$, we have $f=\chi\tilde f$, where $\tilde f$ is the 
periodization of $f$. The estimate~\e{n471} then follows from the
boundedness of the multiplication by $\chi$ from $H^\sigma(\xT^2)$
into $H^\sigma(\xR^2)$, and of the periodization 
from the functions of $H^\sigma(\xR^{2})$ supported in $\supp\chi$
into $H^\sigma(\xT^{2})$.
\end{proof}

The end of this section is devoted to the proof of Proposition~\ref{PK:7.1}.
We proceed by contradiction and suppose that, for all $f_0 \in \bar{B}_{X^s(M)}(\omega_0, r)$,
\be\label{n745}
\sup_{t \ge 0} \lA f_0 \circ \varphi_t^{f_0}\rA_{H^s(M)} \leq N.
\ee
Henceforth $\Phi_t=\Phi_t^{\omega_0}$ and $\varphi_t=\varphi_t^{\omega_0}$ denote the flow associated with $\omega_0$ 
and its inverse. For $f_0$ in this ball, denote by $u^{f_0}$ the velocity of the 
solution with initial vorticity $f_0$. It follows
from~\e{n745} and Proposition~\ref{P203} that $\sup_t\lA u^{f_0}(t)\rA_{H^{s +1}}\le CN$
when $M=\xT^2$, and $\sup_t\lA u^{f_0}(t)\rA_{H^{s +1}}\le C(N+\lA \omega_0\rA_{X^{s}}+r)$ 
when $M=\xR^2$, where we used conservation of energy, 
$\lA u^{f_0}(t)\rA_{L^2}=\lA f_0\rA_{\dot H^{-1}}$. These bounds are uniform in $f_0$.

\subsection{Construction of the perturbation}
We seek an initial perturbation $\eta$ such that $\lA\eta\rA_{X^s(M)}\le r/2$ and, for some $T>0$,
\begin{equation}\label{n:7.2}
\lA \eta\circ \varphi_T \rA_{H^s(M)}\geq 5 N.
\end{equation}

By \e{n:1201} and the continuity of $D\varphi_T$, 
there exist $T> 0$ and $y_0\in M$ such that 
\be\label{n:7.10}
\la A\ra>\Big(C_0\frac{N}{r}\Big)^{1/s},\qquad A\defn D\varphi_T(y_0)=\left(D\Phi_T(x_0)\right)^{-1},\qquad x_0\defn\varphi_T(y_0), 
\ee
where we recall that $\la A\ra$ is the operator norm and $C_0\ge 1$ is a constant, 
depending only on $s$, $M$ and on the profile chosen below
independently of $T$, which is fixed in \e{n302}. Since $\la A^\top\ra=\la A\ra$, there exists
a unit vector $e$ such that 
\be\label{n472} 
\la A^\top e\ra=\la A\ra ,
\ee
namely, $e$ is a maximizing direction for $A^\top$. Notice that $\la A^{-1}\ra\le\lA D\Phi_T\rA_{L^\infty}$ and
$\la A\ra\le \lA D\Phi_T\rA_{L^\infty}$. 

Now that $T$, $x_0$ and $e$ are fixed, we introduce 
a high-frequency profile $\eta$ as follows. 
Given a radial cut-off function $\theta\in C_c^\infty(\xR^2)$ 
satisfying $\theta(z)=1$ for $\la z\ra\le 1/2$ and $\theta(z)=0$ for $\la z\ra\ge 1$, we 
fix a frequency $\mu_0\ge 4$, depending only on $\theta$, such that
\be\label{n473}
\int_{\la\xi\ra\le \mu_0/2}\bla \hat\theta(\xi)\bra^2\dxi\ \ge\ \frac34\blA \hat\theta\brA_{L^2}^{2} 
\qquad\text{and}\qquad
\sup_{\la\xi\ra \ge 3\mu_0/2}\bla\hat\theta(\xi)\bra^2\ \le\ \frac{\blA\hat\theta\brA_{L^{2}}^2}{100\,\la B(0,\mu_0/2)\ra}.
\ee
Such a choice is possible because $\hat\theta$ is a Schwartz function. We set
$$
G(z)= \theta(z)\sin(\mu_0z\cdot e),\qquad 
\eta(x) = \frac{r}{C_1}\,\delta^{s-1}\,G\Big(\frac{x-x_0}{\delta}\Big),
$$
where $0< \delta \le \delta_0$ is a small parameter, and $\delta_0$ is such that 
$4\lA D\Phi_T\rA_{L^\infty}\delta_0\le 1$. On $\xT^2$, we identify the balls of radius at most $1$ with 
Euclidean balls. The constant $C_1=C_1(\theta, \mu_0, s, M)$ is
chosen as follows. Since $\theta$ is radial, 
the norms of $G$ do not depend on $e$. By scaling and Lemma \ref{L419}, we have
\be\label{n818}
\lA \eta\rA_{H^\sigma(M)} \le C\,\frac{r}{C_1}\,\delta^{s-\sigma}\qquad(0\le \sigma \le s+1, \ \delta\le 1).
\ee
When $M=\xR^2$, the norm of $X^s(\xR^2)$ also involves the $\dot H^{-1}$
norm, and $\lA \eta\rA_{\dot H^{-1}}=\frac{r}{C_1}\delta^{s+1}\lA G\rA_{\dot H^{-1}}$, which
is finite since $G$ has zero integral. We choose $C_1$ such that
$\lA\eta\rA_{X^{s}(M)}\le r/2$ for $\delta\le 1$. Since $\theta$ is even and sine is
odd, $\eta$ has zero average, and hence $\omega_0+\eta\in \bar{B}_{X^s(M)}(\omega_0, r)$. We will also use 
\be\label{n474}
\lA \eta\rA_{H^{s+1}}\le Cr\,\delta^{-1},\qquad
\lA \eta\rA_{W^{1,\infty}}\le Cr\,\delta^{s-2}
\ee
(the first bound is~\e{n818} with $\sigma=s+1$, and the second one
follows by scaling) and, for $1\le \sigma\le s+1$,
\be\label{n301}
\lA \nabla^{\perp}\Delta^{-1}(\eta)\rA_{H^\sigma}\le C\,r\,\delta^{s+1-\sigma}, 
\ee
where~\e{n301} follows from \e{n818} and Proposition~\ref{P203}.

\begin{lemma}\label{L:7.4}
There exists $\delta_1>0$, depending on $T,N,r,\omega_0,\theta,s$ and $M$, such that, for 
$0<\delta\le\min(\delta_0,\delta_1)$, the estimate \e{n:7.2} holds.
\end{lemma}

\begin{proof}
Consider the ball $ \Omega\defn B(y_0, 2\lA D\Phi_T\rA_{L^\infty}\delta)$.
On this ball, we approximate the 
inverse flow $\varphi\defn \varphi_T$ by the first-order Taylor expansion of $\varphi$ at $y_0$:
$$
\varphi_{\lin}(y)\defn x_0+A(y-y_0). 
$$
We have 
\begin{align*}
&\varphi^{-1}(B(x_0, \delta))=\Phi_T(B(x_0,\delta))\subset B(y_0,\lA
D\Phi_T\rA_{L^\infty}\delta),\\
&\varphi_{\lin}^{-1}(B(x_0, \delta))=y_0+A^{-1}B(0,\delta)\subset B(y_0, \lA
D\Phi_T\rA_{L^\infty}\delta). 
\end{align*}
Hence, both $\eta\circ\varphi$ and
$\eta\circ\varphi_{\lin}$ are supported in $B(y_0, \lA D\Phi_T\rA_{L^\infty}\delta)$, which 
is contained in $\Omega$. We write $\eta\circ\varphi=\eta\circ\varphi_{\lin}+\mathcal R$. 

We first establish the lower bound for the linear part.
Lemma~\ref{L419} and the scaling
$\lA g(\cdot/\delta)\rA_{\dot H^s(\xR^2)}=\delta^{1-s}\lA g\rA_{\dot H^s(\xR^2)}$ 
imply that 
$$
\lA\eta\circ\varphi_{\lin}\rA_{H^{s}(M)}
\ge c_M\lA \eta\circ\varphi_{\lin}\rA_{\dot H^s(\xR^2)} 
=c_M\frac{r}{C_1}\lA G(A\,\cdot)\rA_{\dot H^s(\xR^{2})}.
$$
Since $\det A=1$, the Fourier transform of $G(A\,\cdot)$ is $\hat G(A^{-\top}\cdot)$, and the change 
of variables $k'=A^{-\top}k$ gives 
$$
\lA G(A\,\cdot)\rA_{\dot H^s(\xR^{2})}^2=c\int_{\xR^2}\la A^\top k\ra^{2s}\la\hat G(k)\ra^2\diff k ,\qquad
\hat G(k)=\frac1{2i}\big(\hat\theta(k-\mu_0e)-\hat\theta(k+\mu_0e)\big). 
$$
We restrict the integral to the disk
$D = \{\la k-\mu_0e\ra\le \mu_0/2\}$. For  $k=\mu_0e+\xi\in D$, 
we deduce from \e{n472} that
$$
\la A^\top k\ra \ge \mu_0\la A^\top e\ra-\la A\ra\la\xi\ra=\la A\ra(\mu_0-\la\xi\ra)\ge \frac{\mu_0}2\la A\ra . 
$$
This is where we use that $e$ is a maximizing direction. On $D$, we have
$\la k+\mu_0e\ra\ge 3\mu_0/2$, so that \e{n473} implies 
\begin{align*}
\int_D\bla\hat G\bra^2\diff k&\ge\frac14\int_D\Big(\frac12\la\hat\theta(k-\mu_0e)\ra^2-\la\hat\theta(k+\mu_0e)\ra^2\Big)\diff k\\
&\ge\frac14\Big(\frac38-\frac1{100}\Big)\blA \hat\theta\brA_{L^2}^2 \ge \frac1{12}\blA \hat\theta\brA_{L^2}^2 .
\end{align*}
Hence, $\lA \eta\circ\varphi_{\lin}\rA_{H^{s}(M)}\ge c_2r\,\la A\ra^s$, with 
$c_2=c_2(\theta,\mu_0, s, M)>0$. We fix 
\be\label{n302}
C_0\defn\max\Big(1,\frac{6}{c_2}\Big),
\qquad\text{so that \e{n:7.10} gives}\qquad 
\lA \eta\circ\varphi_{\lin}\rA_{H^s(M)}\ge6N .
\ee

It remains to bound the remainder $\mathcal R$. Since $\omega_0$ belongs to
$H^k(M)$ for all $k$, the inverse flow $\varphi_T$ is smooth by Proposition~\ref{P204}. For $z$ in the
fixed ball $B(0,2\lA D\Phi_T\rA_{L^\infty})$, set
$$
\Psi_\delta(z)=\frac{\varphi_T(y_0 +\delta z)-x_0}{\delta},\qquad
h_\delta(z)=G(\Psi_\delta(z))-G(Az).
$$
By Taylor's formula, $\Psi_\delta$ converges to $z\mapsto Az$ in 
$C^m(\overline{B(0,2\lA D\Phi_T\rA_{L^\infty})})$
for all integers $m\ge 0$. By the support inclusions above, 
$h_\delta$ is supported in
$B(0,\lA D\Phi_T\rA_{L^\infty})$. Extending it
by zero and choosing an integer $m>s$, we deduce that $h_\delta$ 
converges to $0$ in $H^s(\xR^2)$. Since 
$\mathcal R(y)=\frac r{C_1}\delta^{s-1}h_\delta((y-y_0)/\delta)$, it
follows from Lemma~\ref{L419} and scaling that
$$
\lA \mathcal R\rA_{H^s(M)}\le C r\lA h_\delta\rA_{H^s(\xR^2)}\xrightarrow[\delta\to0]{}0.
$$
We choose $\delta_1$ such that this norm is at most $N$ for $\delta\le \delta_1$. Then
$$
\lA \eta\circ\varphi_T\rA_{H^s(M)}\ge \lA \eta\circ\varphi_{\lin}\rA_{H^s(M)}-\lA \mathcal R\rA_{H^s(M)} \ge 6N-N=5N,
$$
which concludes the proof.
\end{proof}

\subsection{Stability of the growth}
Set
$$
\omega'_0 \defn \omega_0+\eta.
$$
By construction, $\omega'_0\in\bar{B}_{X^s(M)}(\omega_0, r)$. Let $u$, $u'$ be the solutions 
with initial vorticities $ \omega_0$, $\omega'_0$, 
and let $\varphi$, $\varphi'$ be their inverse flows. At time $T$, we have $\omega'(T)=\omega'_0\circ\varphi'_T$, which we decompose as 
\begin{align*}
\omega'(T)=\
&\underbrace{\omega_0\circ\varphi_T}_{=(1)}
+\underbrace{\omega_0\circ\varphi'_T-\omega_0\circ\varphi_T}_{=(2)}
+\underbrace{\eta\circ\varphi'_T-\eta\circ\varphi_T}_{=(3)} 
+\underbrace{\eta\circ\varphi_T}_{=(4)} . 
\end{align*}
By \e{n745} and Lemma \ref{L:7.4}, we have, respectively, $\lA (1)\rA_{H^s}\le N$ and $\lA (4)\rA_{H^s}\ge 5N$.

It remains to control the error terms (2) and (3).
Henceforth, $C_T$ denotes constants depending on $T,N,r,\omega_0,\theta,s$ 
and $M$, but not on $\delta$. By the bounds on $u^{f_0}$ deduced from \e{n745}, we have
$$
\mathcal M\defn\sup_{t\in[0,T]}\left(\lA u(t)\rA_{H^{s+1}}+\lA u'(t)\rA_{H^{s+1}}\right)\le C_T . 
$$

We first compare the two inverse flows. Since $u'_0-u_0=\nabla^{\perp}\Delta^{-1}(\eta)$, Lemma \ref{L489} and~\e{n301}
imply that, for $1\le \sigma\le s$, 
$$
\lA \varphi'_T-\varphi_T\rA_{H^\sigma}\le C_T\lA \nabla^{\perp}\Delta^{-1}(\eta)\rA_{H^\sigma}\le C_T\, r\,\delta^{s+1-\sigma}.
$$
Fix $0 < \nu_0<\min(1,s-1)$. Applying the previous estimate with $\sigma = s$ and $\sigma=1+ \nu_0$,
and using the Sobolev embedding $H^{1+\nu_0}(M)\subset L^\infty(M)$, we get
\be\label{n303}
\lA \varphi'_T-\varphi_T\rA_{H^s}\le C_T r\,\delta,
\qquad 
\lA \varphi'_T-\varphi_T\rA_{L^\infty} \le C_T r\,\delta^{s-\nu_0}. 
\ee
We next check the assumptions of Lemma~\ref{L488}, applied 
with the regularity index $s$, for $a=\varphi'_T$ and $b=\varphi_T$. Lemma \ref{L487} 
shows that these maps are bi-Lipschitz with a constant $L=C_T$, 
and that $\lA Da-I\rA_{H^s}+\lA Db-I\rA_{H^{s}}\le C_T$. Set $\theta_0=\min(1,s-1-\nu_0)$, so that $s-\theta_0 > 1$.
By interpolation between $H^{s-1}$ and $H^s$ and by Sobolev embedding, we have 
$$
\lA D(a-b)\rA_{L^\infty}
\le C\lA D(a-b)\rA_{H^{s-1}}^{\theta_0}\lA D(a-b)\rA_{H^s}^{1-\theta_0}
\le C_T(r\delta)^{\theta_0},  
$$
and the smallness assumption of Lemma~\ref{L488} holds for $\delta$ small enough.

To estimate the term $(2)$, we apply Lemma~\ref{L488} with
$f= \omega_0$, which belongs to 
$H^{s+1}\cap W^{1,\infty}$. Using \e{n303}, we obtain

$$
\lA(2)\rA_{H^s} 
\le C_T\Big(\lA\omega_{0}\rA_{H^{s+1}}\lA\varphi'_T-\varphi_T\rA_{L^\infty} 
+\lA\omega_0\rA_{W^{1,\infty}}\lA\varphi'_T-\varphi_{T}\rA_{H^s}\Big)
\le C_{T}\, r\big(\delta^{s-\nu_{0}}+\delta\big).
$$
The estimate of $(3)$ is similar. Applying Lemma~\ref{L488}
with $f=\eta$, and using~\e{n474} and~\e{n303}, 
we find that 
$$
\lA(3)\rA_{H^s}
\le C_T r\Big(\delta^{-1}\lA \varphi'_T-\varphi_T\rA_{L^\infty}
+\delta^{s-2}\lA \varphi'_T-\varphi_T\rA_{H^s}\Big)
\le C_T r^{2}\left(\delta^{s-1-\nu_0}+\delta^{s-1}\right). 
$$

Since $\nu_0<s-1$, the error terms $(2)$ and $(3)$ tend to  $0$ as $\delta$ goes to
$0$. We choose $\delta\le\min(\delta_0,\delta_1)$
so small that the assumptions of Lemma~\ref{L488} hold and 
$\lA (2)\rA_{H^{s}}+\lA (3)\rA_{H^s}\le N$. Then 
$$
\lA \omega'(T)\rA_{H^s(M)}
\ge\lA (4)\rA_{H^s}-\lA(1)\rA_{H^s}-\lA (2)\rA_{H^s}-\lA (3)\rA_{H^s}
\ge5N-N-N=3N . 
$$
Since $\omega'_0\in\bar{B}_{X^s(M)}(\omega_0,r)$, this contradicts \e{n745}, which completes the proof of Proposition~\ref{PK:7.1}.

\section{Generic unboundedness in \texorpdfstring{$H^s$}{Hs}}\label{S:Baire}

We prove the $H^s$ assertion of Remark~\ref{R:Hs} and Theorem~\ref{T:R2}. To pass from the growth
given by Proposition \ref{PK:7.1} to generic unboundedness, we use
a Baire category argument introduced by
Hani~\cite{zbMATH06303378} for dispersive equations (see 
also~\cite{zbMATH08109714,zbMATH08101466}). 

\subsection{Genericity on \texorpdfstring{$\xT^2$}{T2}: proof of Remark~\ref{R:Hs}}\label{S474}

Denote by $S(t)$ the solution map of Section \ref{S:classique}, so that $S(t)\omega_0=\omega(t)$, and set 
$$
\mathcal{B}\defn\left\{\omega_0\in H^s_0(\xT^2) :\limsup_{t \to +\infty} \lA S(t)\omega_0 \rA_{H^{s}}=+\infty \right\}.
$$
Remark~\ref{R:Hs} asserts that $\mathcal B$ is a dense $G_\delta$ subset of $H^s_0(\xT^2)$. For $N\in \xN^*$, set 
$$
\mathcal{U}_N\defn \left\{ \omega_0 \in H^{s}_0(\xT^{2}) : \sup_{t\ge 0} \lA S(t)\omega_0 \rA_{H^s} > N\right\}.
$$
By Theorem \ref{T:classique}, $t\mapsto\lA S(t)\omega_0\rA_{H^s}$ is continuous
and hence bounded on compact intervals.
Thus $\mathcal{B}=\bigcap_{N \in \xN^*} \mathcal{U}_N$. Since $H^s_0(\xT^2)$ is 
complete, it follows from Baire's theorem that $\mathcal B$ 
is a dense $G_\delta$ set as soon as each $\mathcal U_N$ is open and dense. 

\begin{lemma}\label{L420} 
For all $N \in \xN^*$, the set $\mathcal{U}_N$ is open in $H^s_0(\xT^2)$.
\end{lemma}
\begin{proof}
Let $\omega_0 \in \mathcal{U}_N$, and fix
$t_0\ge 0$ such that $\lA S(t_0)\omega_0 \rA_{H^s}>N$. By Theorem~\ref{T:classique}
and the comments following it, the map
$S(t_0)$ is continuous on $H^s_0(\xT^2)$. Hence the set 
$$
\mathcal V\defn \big\{\omega\in H^s_0(\xT^2):\lA S(t_0)\omega\rA_{H^s}>N\big\}
$$
is open. Since $\omega_0\in\mathcal V\subset\mathcal U_N$, this proves that $\mathcal{U}_N$ is open. 
\end{proof}

\begin{lemma}\label{L449} 
For all $N \in \xN^*$, the set $\mathcal{U}_N$ is dense in $H^s_0(\xT^2)$.  
\end{lemma}
\begin{proof}
Let $f \in H^s_0(\xT^2)$ and $\delta>0$. We seek $h \in \mathcal{U}_N$ such that $\lA f-h \rA_{H^s}<\delta$. 

Since smooth functions with zero mean are dense in $H^s_0(\xT^2)$, there exists
$g\in C^\infty_0(\xT^{2})$ such that $\lA f-g\rA_{H^s}<\delta/2$. If $g$ vanishes
identically, we replace it by $\eps\cos x_1$ with $\eps$ small enough, so that 
one can assume that $g\neq0$. By Remark~\ref{R108}, $g$ satisfies the
assumption~\e{n:1201} of Proposition~\ref{PK:7.1}. Applying this
proposition with $M=\xT^2$, $\omega_0=g$ and $r=\delta/2$, we obtain
$h \in \bar{B}_{H^s_0}(g, \delta/2)$ such that $\sup_{t\ge 0} \lA S(t)h \rA_{H^{s}}>N$.
Then $h\in\mathcal{U}_N$, and $\lA f-h\rA_{H^s}\le\lA f-g\rA_{H^s}+\lA g-h\rA_{H^s}<\delta$.
\end{proof}

\subsection{Genericity on \texorpdfstring{$\xR^2$}{R2}: proof of Theorem~\ref{T:R2}}\label{S506}

We argue as on $\xT^2$, in the complete space $X^s(\xR^2)$ defined in \e{n103}. It suffices to prove that, for all $N\in\xN^*$, the set 
$$
\mathcal U_N^{\xR^2}\defn \Big\{\omega_0\in X^s(\xR^2):\sup_{t\ge 0}\lA S(t)\omega_0\rA_{H^s(\xR^2)}>N\Big\}
$$
is open and dense in $X^{s}(\xR^2)$. As in the proof of Lemma~\ref{L420}, the openness follows from
the continuity of $S(t_0)$ on $X^s(\xR^2)$ (see Theorem~\ref{T:classique}
and the comments following it). 

To prove the density, we begin by approximating the elements of $X^s(\xR^2)$ by 
smooth compactly supported functions with zero integral. 
Let $f\in X^s(\xR^2)$ and $\delta>0$. Since the norm of $X^s(\xR^2)$ is a weighted $L^2$ norm of $\hat f$, we 
may approximate $f$ by a Schwartz function
$f_1$ whose Fourier transform vanishes near the origin. In particular, $f_1$ has zero integral. 
Fix $\chi\in C^\infty_c(\xR^2)$ equal to $1$ near the origin, and $\vartheta\in C^\infty_c(\xR^2)$
with integral $1$, and set
$$
f_R=\chi(\cdot/R)f_1-c_R\vartheta, \qquad c_R=\int_{\xR^2}\chi(x/R)f_1(x)\dx. 
$$
We have $\int_{\xR^2}f_R\dx=0$ and, as $R\to+\infty$,
$$
c_R\longrightarrow0,\qquad f_R\longrightarrow f_1\quad\text{in }H^s(\xR^2), \qquad 
\int_{\xR^2}\la x\ra\la f_R-f_1\ra\dx\longrightarrow0.
$$
For a function $w$ with zero integral, $\la \hat w(\xi)\ra\le \la\xi\ra\int\la x\ra\la w(x)\ra\dx$, and hence 
$\lA w\rA_{\dot H^{-1}}$ is controlled by $\lA w\rA_{L^{2}}$ and by the first moment $\int\la x\ra\la w(x)\ra\dx$. 
Consequently, $f_R$ converges to $f_1$ in $X^s(\xR^2)$, and there exists $g\in C^\infty_c(\xR^2)$ with zero 
integral such that $\lA f-g\rA_{X^s}< \delta/2$. 
If $g$ vanishes identically, we replace it 
by $\eps\Delta\chi$ with $\eps$ small enough, so that one can assume that $g\neq0$. 

As on $\xT^2$, Remark~\ref{R108} shows that 
$g$ satisfies the assumption \e{n:1201} of
Proposition~\ref{PK:7.1}. Applying this proposition with $M=\xR^2$, $\omega_0=g$ and $r=\delta/2$, we obtain 
$h\in X^{s}(\xR^2)$ such that $\lA h-g\rA_{X^s}\le \delta/2$ and $\sup_{t\ge 0}\lA S(t)h\rA_{H^s}>N$. 
Then $h\in\mathcal U_N^{\xR^{2}}$ and $\lA f-h\rA_{X^s}<\delta$, which 
concludes the proof of Theorem~\ref{T:R2}.

\section{Growth of the vorticity gradient on \texorpdfstring{$\xT^2$}{T2}}
\label{S507}

In this section we prove Theorem \ref{T11}, together with the
following analogue in $C^1_0(\xT^2)$ mentioned in 
Remark~\ref{R:Hs}. As in Section~\ref{S:Baire}, $S(t)$ denotes
the map which sends the initial vorticity to the
vorticity at time $t$ (for initial data 
in $W^{1,p}_0(\xT^2)$ or $C^1_0(\xT^2)$, it is well defined by
Lemma~\ref{L450} below).

\begin{theorem}\label{T106} 
The set 
$$
\left\{f\in C^1_0(\xT^2):
\limsup_{t\to+\infty}\lA \nabla S(t)f\rA_{L^\infty}=+\infty
\right\}
$$
is a dense $G_\delta$ subset of $C^{1}_0(\xT^2)$. 
\end{theorem}

Following the previous scheme, 
we first solve the Cauchy problem, and then perturb a smooth initial datum 
to make the gradient of the vorticity large at some time, before concluding
with Baire's theorem. We take $$
\eta=a\delta^{1-2/p}G_e\big((x-x_0)/\delta\big),
$$
where $G_e$ is a fixed smooth
function supported in the unit ball (see \e{n202}).
The norms of $\eta$ in $W^{1,p} $ and in
$H^{2-2/p}(\xT^2)$ are bounded independently of $\delta$ (the stability 
estimates of Appendix~\ref{S508} apply in $H^{2-2/p}(\xT^2)$
since $2-2/p>1$).

To treat the two cases together, we set $\mathcal{E}_p\defn W^{1,p}_0(\xT^2)$ for $2<p<+\infty$ and
$\mathcal{E}_\infty\defn C^1_0(\xT^2)$, with their usual norms, and use the convention 
$1/p=0$ when $p=+\infty$. Hereafter, $2<p\le +\infty$ is fixed, and 
the constants may depend on $p$. 

%\subsection{The Cauchy problem} 

We recall the standard well posedness in $\mathcal{E}_p$. 

\begin{lemma}\label{L450}
For all $f\in \mathcal{E}_p$, the Euler equations have a unique global 
solution, with zero-mean velocity, whose vorticity satisfies 
$$
\omega\in C([0,+\infty);\mathcal{E}_p),\qquad \omega(0)=f.
$$
Moreover, for all $T>0$, the map 
$$
\mathcal{E}_p\ni f\longmapsto S(\cdot)f \in C([0,T];\mathcal{E}_p)
$$
is continuous. 
\end{lemma}

\subsection{A perturbation with a prescribed gradient} 

The following proposition plays the role of Proposition \ref{PK:7.1}.

\begin{proposition}\label{P202}  
Let $g\in C^\infty(\xT^2)$ be non-zero with zero mean. For 
all $\eps > 0$ and all $N>0$, there exist $T> 0$ and a smooth
zero-mean function $h$ such that 
$$
\lA h-g\rA_{\mathcal{E}_p}<\eps, 
\qquad
\lA \nabla S(T)h\rA_{L^p}>N.
$$
\end{proposition}

\begin{proof}
We will apply the results of Appendix~\ref{S508} with 
$$
s=2-\frac2p, 
$$
which satisfies $s>1$ since $p>2$.  

\smallbreak
\noindent\emph{Step 1: the perturbation.} 
Fix a radial function $\chi\in C^\infty_c(B(0, 1))$ and $0<r_0 < 1$ 
such that $\chi=1$ on $B(0,r_0)$. For 
a unit vector $e\in \xR^2$, the function $G_e(z)=\chi(z)(z\cdot e)$ satisfies $\nabla G_e=e$ on $B(0,r_0)$ and
has zero integral, since it is odd (recall that $\chi$ is radial). 
Given $x_0\in \xT^{2}$, $a>0$ and $0<\delta\le 1$, we 
denote by $\eta_\delta$ the periodization of the function 
\be\label{n202}
x\mapsto a\delta^{1-2/p} G_e\left(\frac{x-x_0}{\delta}\right).
\ee
By scaling, we have
\begin{align}
&\int_{\xT^2}\eta_\delta\dx=0,\qquad
\nabla\eta_\delta=a\delta^{-2/p}e 
\quad\text{on }B(x_0, r_0\delta), 
\label{n203}\\
&\lA \eta_\delta\rA_{\mathcal{E}_p}\le Ca,
\qquad 
\lA D^2\eta_\delta\rA_{L^\infty}
\le Ca\delta^{-1-2/p},
\label{n212}\\
&\lA \eta_\delta\rA_{H^q}  
\le C_qa\delta^{s-q}
\qquad(0\le q\le s),
\label{n213}
\end{align}
where $C$ and $C_q$ depend only on $\chi$ (and on $q$ for $C_q$), and not on $e$, $x_0$, $a$ or
$ \delta$. To prove \e{n213}, one can
interpolate between the bounds $\lA \eta_\delta\rA_{L^2}\le Ca\delta^s$ and
$\lA \eta_\delta\rA_{H^2}\le Ca\delta^{s-2}$ given by scaling, and in particular  $\eta_\delta$ is bounded in 
$H^{s}$ uniformly in $\delta$. We now fix 
$0<a\le 1$ such that $Ca<\eps$ in~\e{n212},
so that $\lA \eta_\delta\rA_{\mathcal{E}_p}<\eps$. 

\smallbreak
\noindent\emph{Step 2: amplification.} 
Denote by $\Phi_t$ and $\varphi_t$ the flow and the inverse flow of the solution 
with initial vorticity $g$, for which Theorem~\ref{T105}
and~\e{n106} imply that $\lA D\varphi_t\rA_{L^\infty}$ is unbounded for  $t\ge 0$. We
may therefore choose $T >0$ and $y_0\in\xT^2$ such that
$$
A\defn D\varphi_T(y_0)\quad\text{satisfies}\quad c_pa\la A\ra>2N,
$$
where $\la A\ra$ is the operator norm. As in Section~\ref{S:Koch}, we
choose a unit vector $e$ such that $\la A^\top e\ra=\la A\ra$ 
and set $x_0=\varphi_T(y_0)$. The parameters $a,T,x_0,e$ are now fixed, and only $\delta$ remains free.

When $p= +\infty$, the identity $\nabla\eta_\delta(x_0)=ae$ and the chain rule imply that
$$
\nabla(\eta_\delta\circ\varphi_T)(y_0)=aA^\top e,
$$
a vector of norm $a\la A\ra$. When $p<+\infty$, we consider the set 
$$
E_\delta=\Phi_T\left(B(x_0,r_0\delta)\right), 
$$
whose measure is $\la E_\delta\ra=\la B(0,r_0)\ra\delta^2$
(recall that $\Phi_T$ preserves area). Since $\Phi_T(x_0)=y_0$ 
and $D\varphi_T$ is continuous, we have, for $\delta$ small enough, 
$$
\la D\varphi_T(y)^\top e\ra
\ge\frac{1}{2}\la A\ra
\quad\text{for all }y\in E_\delta.
$$
On the other  hand, the chain rule and~\e{n203}
imply that, for all $y\in E_\delta$, 
$$
\nabla(\eta_\delta\circ\varphi_T)(y)
=a\delta^{-2/p}D\varphi_T(y)^\top e. 
$$
Integrating over $E_\delta$ when $p<+\infty$, and using the previous computation when $p= +\infty$, we obtain, for $\delta$ small enough, 
\be\label{n214} 
\lA \nabla(\eta_\delta\circ\varphi_T)\rA_{L^{p}}
\ge c_pa\la A\ra>2N.
\ee

\smallbreak
\noindent\emph{Step 3: comparison of the flows.}
Denote by $\omega^\pm$, $u^\pm$ and $\varphi_t^\pm$ the vorticities, the 
velocities and the inverse flows of the
solutions with initial vorticities $g\pm\eta_\delta$. Assertion~(3)
of Theorem~\ref{T:classique} and Proposition \ref{P203} imply
that the velocities $u$, $u^+$ and $u^-$ are 
bounded in $C([0,T];H^{s+1})$ uniformly in
$\delta$, because these initial vorticities are bounded in $H^s$ uniformly in $\delta$ 
by \e{n213}. Hereafter, $C_T$ denotes constants depending only on $g$, $a$, $T$ and $\chi$, and not on $\delta$. 
We claim that 
\begin{align}
&d_0^\pm\defn \lA \varphi_T^\pm-\varphi_T\rA_{L^\infty}
\le C_Ta\delta^{\frac32-\frac1p},
\label{n215}\\
&d_1^\pm\defn \lA D\varphi_T^\pm-D\varphi_T\rA_{L^\infty}
\le C_T(a\delta)^{\frac12-\frac1p}.
\label{n241}
\end{align}
Since $p>2$, $d_1^\pm+ \delta^{-1}d_0^\pm\to0$ as $\delta\to0$.

To prove \e{n215},
consider $r\in[1,s]$. Starting from $u^\pm(0)-u(0)=\pm\nabla^{\perp}\Delta^{-1}(\eta_\delta)$, we deduce
from Lemma~\ref{L489} and \e{n213} that
\be\label{n242}
\lA \varphi_T^\pm-\varphi_T\rA_{H^r}  
\le C_T\lA \nabla^{\perp}\Delta^{-1}(\eta_\delta)\rA_{H^r}
\le C_T\lA \eta_\delta\rA_{H^{r-1}}
\le C_Ta\delta^{s+1-r}.
\ee
Taking $r=\frac32-\frac1p$ (which satisfies $s+1-r=\frac32-\frac1p$
and belongs to $[1,s]$ since $p>2$), we get~\e{n215}
from the Sobolev embedding $H^r(\xT^2)\subset L^\infty(\xT^2)$,
valid for $r>1$. 

To prove \e{n241}, we take
$r=s$ in \e{n242} and find that
$$
\lA D\varphi_T^\pm-D\varphi_T\rA_{H^{s-1}}
\le C_Ta\delta.
$$
On the other hand, it follows from \e{n204} that
$\lA D\varphi_T^\pm-D\varphi_T\rA_{H^s}\le C_T$. Writing 
$\frac32-\frac1p=\big(\frac12-\frac1p\big)(s-1)+\big(\frac12+\frac1p\big)s$, we deduce 
from the interpolation inequality that 
\begin{align*}
\lA D\varphi_T^\pm- D\varphi_T\rA_{H^{\frac32-\frac1p}}
&\le
\lA D\varphi_T^\pm- D\varphi_T\rA_{H^{s-1}}^{\frac{1}{2}-\frac1p}
\lA D\varphi_T^\pm- D\varphi_T\rA_{H^s}^{\frac{1}{2}+\frac1p}\\
&\le C_T(a\delta)^{\frac{1}{2}-\frac1p}.
\end{align*}
The Sobolev embedding 
$H^{\frac32-\frac1p}(\xT^2)\subset L^\infty(\xT^{2})$ then
implies~\e{n241}. 

\smallbreak
\noindent\emph{Step 4: comparison of the transported gradients.}
By the chain rule, we have 
\begin{align*}
&\nabla(\eta_\delta\circ\varphi_T^\pm)-\nabla(\eta_\delta\circ\varphi_T)\\ 
&\quad=(D\varphi_T^\pm-D\varphi_T)^\top (\nabla\eta_\delta)\circ\varphi_T^\pm\\
&\qquad\quad+(D\varphi_T)^\top \bigl((\nabla\eta_\delta)\circ\varphi_T^\pm -(\nabla\eta_\delta)\circ\varphi_T\bigr). 
\end{align*}
The $L^p$ norm of the first term is bounded by 
$d_1^\pm\lA \nabla\eta_\delta\rA_{L^{p}}$
(because the map $\varphi_T^\pm$ preserves area), and hence by $Cad_1^\pm$ according to \e{n212}. 

To estimate the second term, we use the mean value inequality and~\e{n212} to get
$$
\la(\nabla\eta_\delta)\circ\varphi_T^\pm -(\nabla\eta_\delta)\circ\varphi_T\ra\le Ca\delta^{-1-2/p}d_0^\pm. 
$$
When $p= +\infty$, this bounds the $L^\infty$ norm by $Ca\delta^{-1}d_0^\pm$. 
When $p<+\infty$, we 
also use the fact that this difference vanishes outside 
$$
(\varphi_T^\pm)^{-1}(B(x_0, \delta)) \cup\varphi_T^{-1}(B(x_0,\delta)),
$$
whose measure is at most $2\la B(0,1)\ra\delta^2$, so that 
the $L^p$ norm of this difference is also bounded
by $Ca\delta^{-1}d_0^\pm$. Recalling that $D\varphi_T$ is bounded, we conclude that 
\be\label{n259}
\lA \nabla(\eta_\delta\circ\varphi_T^\pm)-\nabla(\eta_\delta\circ\varphi_T)\rA_{L^p}  
\le C_Ta\left(d_1^\pm+\delta^{-1}d_0^\pm\right)  
\longrightarrow0. 
\ee
Similarly, since $g$ is smooth and does not depend on $\delta$, 
we deduce from the chain rule that
$$
\lA \nabla(g\circ\varphi_T^\pm)-\nabla(g\circ\varphi_T)\rA_{L^p}\le C_T(d_1^\pm + d_0^\pm)\longrightarrow 0.
$$

\smallbreak
\noindent\emph{Step 5: choice of the sign.}
Set 
$$
b=\nabla(g\circ\varphi_T),\qquad v_\delta=\nabla(\eta_\delta\circ\varphi_T).
$$
Writing $\omega^\pm(T)=(g\pm\eta_\delta)\circ\varphi_T^\pm $, we deduce from Step~4 that
$$
\nabla\omega^\pm(T)=b\pm v_\delta+R_\delta^\pm, \qquad\text{where}\quad \lA R_\delta^\pm\rA_{L^p}\longrightarrow0.
$$
Since $2v_\delta=(b+v_\delta)-(b-v_\delta)$, we have  
$$
\max\bigl(\lA b+v_\delta\rA_{L^p},  \lA b-v_\delta\rA_{L^p}\bigr) \ge \lA v_\delta\rA_{L^{p}},
$$
where the right-hand side is greater than
$2N$ by~\e{n214}. Choosing $\delta$ small enough that  $\lA R_\delta^+\rA_{L^{p}}<N$ and $\lA R_\delta^-\rA_{L^p}<N$,
we find that $\lA \nabla\omega^\pm(T)\rA_{L^p}> N$ for at least one sign, and the
corresponding function $h =g\pm\eta_\delta$ is smooth, has zero mean, and satisfies $\lA h-g\rA_{\mathcal{E}_p}<\eps$ by Step~$1$. 
This completes the proof.
\end{proof}

\subsection{Baire's theorem}

We now conclude the proofs of Theorems \ref{T11} and~\ref{T106}, as in Section~\ref{S474}.

\begin{proof}[Proof of Theorems~\ref{T11} and \ref{T106}] 
For $N\in \xN^*$, set
$$
\mathcal U_N\defn\left\{f\in \mathcal{E}_p: 
\sup_{t \ge 0}\lA \nabla S(t)f\rA_{L^p}>N\right\}.  
$$
By Lemma~\ref{L450}, for all $f\in \mathcal{E}_p$, the function $t\mapsto\lA \nabla S(t)f\rA_{L^p}$
is continuous, and hence bounded on compact intervals. Consequently,
$$
\bigcap_{N\in\xN^*}\mathcal U_N =\left\{f\in \mathcal{E}_p\,:\, \limsup_{t\to+\infty}\lA \nabla S(t)f\rA_{L^p}=+\infty\right\}.
$$
Baire's theorem, applied in the Banach space $\mathcal{E}_p$, implies that this set is a dense
$G_\delta$ subset of $\mathcal{E}_p$ as soon as each $\mathcal U_N$ is open and dense.

Let $f\in \mathcal U_N$, and fix $t_0\ge 0$ such that $\lA \nabla S(t_0)f\rA_{L^p} > N$. The continuity of 
the map $S(t_0)$ on $\mathcal{E}_p$ (see Lemma~\ref{L450})  implies that this inequality holds in a neighborhood of $f$,
and $\mathcal U_N$ is therefore open.

To prove that $\mathcal U_N$ is dense, consider $f\in\mathcal{E}_p$, $\eps> 0$, and a non-zero smooth zero-mean 
function $g$ such that $\lA f-g\rA_{\mathcal{E}_p}<\eps/2$. Proposition \ref{P202}, applied with
$\eps/2$, yields $h\in \mathcal U_N$ such that $\lA h-g\rA_{\mathcal{E}_p}<\eps/2$, and hence $\lA h-f\rA_{\mathcal{E}_p}<\eps$.

Theorem~\ref{T11} corresponds to $2< p<+\infty$, and Theorem \ref{T106} to $p=+\infty$. 
\end{proof}

\appendix
\section{The Cauchy problem and the flow map for \texorpdfstring{$s>1$}{s>1}}\label{S508}

In this appendix, we gather the results on the Cauchy problem and on the flow map used in Sections \ref{S:Koch},
\ref{S:Baire} and \ref{S507}. 

The local Cauchy theory is classical 
(see~\cite{Kato1967,zbMATH01644218,zbMATH00785410}). For results on Sobolev diffeomorphisms,
we refer to \cite{zbMATH03335835,zbMATH06304950,zbMATH06725184}
(see also~\cite{zbMATH01801548,AlazardSaid-book}).

Throughout this appendix, $s>1$ is a real number, $M$ is either
$\xT^2$ or $\xR^2$, and $H^r_\sigma(M)$ denotes, as above, the space of divergence-free vector fields $u\in H^r(M)^2$
(with zero mean when $M=\xT^2$). We set $X^{r}(\xT^2)=H^r_0(\xT^2)$ and define $X^r(\xR^2)$ as 
in \e{n103} (with  $s$ replaced by
$r$). On $\xT^2$, we identify maps of the form  $\Id$ plus a periodic map with their lifts to 
$\xR^2$, so that their displacements and their differences are periodic functions.
We write $A\les B$ when $A\le KB$ for some constant $K$ depending only on 
$M$ and on the Sobolev indices, and $C(\cdot)$ denotes various non-decreasing functions depending 
only on parameters that are considered fixed. 

\subsection{Biot--Savart law and compositions}

We begin with the Biot--Savart  law, which relates the Sobolev regularity of the velocity to that of the vorticity. 

\begin{proposition}\label{P203}
Let $r \ge0$. The map $u\mapsto\curl u$ is an isomorphism from $H^{r+1}_\sigma(M)$ onto $X^r(M)$.
Its inverse, denoted by $\omega\mapsto\nabla^{\perp}\Delta^{-1}(\omega)$, is the Fourier multiplier of symbol \e{BS}.
On $\xR^2$, we have $\lA \nabla^{\perp}\Delta^{-1}(\omega)\rA_{L^2}=\lA \omega\rA_{\dot H^{-1}}$.
\end{proposition}

%This follows directlyt from Plancherel's identity. 
We next recall the classical product estimates.

\begin{lemma}\label{L485}
Let $\sigma \ge0$ and $s>1$. For all $f,g\in H^\sigma(M)\cap L^\infty(M) $, we have
$$
\lA fg\rA_{H^\sigma}\les_\sigma \lA f\rA_{H^\sigma}\lA g\rA_{L^\infty}+\lA f\rA_{L^\infty}\lA g\rA_{H^\sigma}. 
$$
Moreover, we have 
\begin{align}
\lA fg\rA_{H^q}&\les_{q,s}\lA f\rA_{H^s}\lA g\rA_{H^q},
&&0\le q\le s,\label{n260}\\ 
\lA fg\rA_{H^{r-1}}&\les_{r,s}\lA f\rA_{H^{s-1}}\lA g\rA_{H^r},
&&1\le r\le s.\label{n261}
\end{align}
\end{lemma}

To compare two compositions, we will interpolate between their maps. We need to consider the composition 
estimate for bi-Lipschitz maps because the intermediate maps need not preserve area. 
\begin{lemma}\label{L486} 
Let $L\ge 1$. Consider a $C^1$ diffeomorphism 
$c\colon M\to M$, with a lift of the form $\Id$ plus a periodic 
map when $M=\xT^2$. Assume that $c$ and $c^{-1}$ are $L$-Lipschitz and 
that $\lA Dc-I\rA_{H^s}\le L$. Then, for all $0\le q\le s+1$ and all $f\in H^{q}(M)$, we have 
$$
\lA f\circ c\rA_{H^q}\le C(L)\lA f\rA_{H^q},
$$
where $C(L)$ depends only on $L$, $s$, $q$ and $M$.
For $0\le q\le 1$, the same estimate holds 
under the bi-Lipschitz assumptions alone. 
\end{lemma}
\begin{proof}
For  $q=0$, the change of variables $y= c(x)$ gives the estimate, since $\la\det Dc^{-1}\ra\le L^2$, 
and the case $q=1$ follows by combining this change of variables with the 
chain rule. For $0<q<1$, we use the Gagliardo seminorm
$$
[f]_{H^q}^2=\iint_{\xR^2\times\xR^2} 
\frac{\la f(x)-f(y)\ra^2}{\la x-y\ra^{2+2q}} \dx\dy,
$$
with the periodized kernel on $\xT^2$: after changing both variables,
the bi-Lipschitz bounds control the denominator
and the two Jacobians. 

Finally, for $q>1$, we argue by induction on the integer part of $q $, writing 
$$
D(f\circ c)=(Df)\circ c+((Df)\circ c)(Dc-I), 
$$
where the first term is estimated by the 
induction hypothesis, while for the second term we use \e{n260} with
$q-1\le s$ instead of $q$ (this is where we use the assumption on $Dc-I$). 
\end{proof}

To apply the previous estimate to the flow and to the inverse flow, 
we need bounds on their differentials. 

\begin{lemma}\label{L487}  
Let $T> 0$ and $u\in C([0,T];H^{s+ 1}_\sigma(M))$.
Denote by $\Phi_t$ the flow of $u$ and by $\varphi_t=\Phi_t^{-1}$ the inverse flow,
and set $\mathcal M=\sup_{[0,T]}\lA u(t)\rA_{H^{s+1}}$. 
Then $\Phi_t$ and $\varphi_t$ are $C^1 $ area-preserving diffeomorphisms, and $\Phi-\Id$, $\varphi-\Id$ 
belong to $C([0,T];H^{s+1}(M))$.
Moreover, for all $0\le t\le T$, we have 
\begin{align}
&\lA D\Phi_t\rA_{L^\infty}= \lA D\varphi_t\rA_{L^\infty}\le e^{K\mathcal Mt},\label{n262}\\ 
&\lA D\Phi_t-I\rA_{H^{s}} +\lA D\varphi_t-I\rA_{H^s}
\le C(T,\mathcal M),\label{n204}
\end{align}
where $K$ depends only on $s$ and $M$, and $C(T, \mathcal M)$ depends only on $T$, $\mathcal M$, $s$ and $M$.
\end{lemma}

For the differences of compositions used in Section~\ref{S:Koch},
we integrate along the segment joining the two maps (the smallness 
assumption on $D(a-b)$ ensures that the intermediate maps are bi-Lipschitz).

\begin{lemma}\label{L488}
Let $r > 1$ and $L \ge1$. Consider two $C^1$ maps $a,b$ of $\xR^2$, with $a-\Id$ and $b-\Id$ periodic when $M=\xT^2$. Assume that
$$
L^{-1}\la x-y\ra\le \la a(x)-a(y)\ra\le L\la x-y\ra,
\qquad \lA Da-I\rA_{H^r(M)}\le L,
$$
and that the same bounds hold for $b$, and that $a-b\in H^r(M)$ 
with $\lA D(a-b)\rA_{L^\infty}\le 1/(2L)$. 
Then, for all $f\in H^{r+1}(M)\cap W^{1,\infty}(M)$, we have
$$
\lA f\circ a-f\circ b\rA_{H^r}
\le C(L)\Big(\lA f\rA_{H^{r+ 1}}\lA a-b\rA_{L^\infty} 
+\lA f\rA_{W^{1,\infty}}\lA a-b\rA_{H^r}\Big), 
$$
where $C(L)$ depends only on $r$, $L$ and  $M$.
\end{lemma}
\begin{proof}
Set $h=a-b$ and $c_\theta=b +\theta h$, for
$0\le \theta\le 1$. It follows from the assumptions that, for all $x,y\in \xR^2$,
$$
\frac1{2L}\la x-y\ra\le \la c_\theta(x)-c_\theta(y)\ra 
\le2L\la x-y\ra,
\qquad \lA Dc_\theta-I\rA_{H^r}\le L. 
$$
The lower bound implies that $c_\theta$ is a global diffeomorphism
(it is proper and injective, with an invertible differential at each point), and since on $\xT^2$
it also commutes with the lattice translations,
Lemma~\ref{L486} applies with
$s= r$, uniformly in $\theta$. We then write 
$$
f\circ a-f\circ b=\int_0^{1}(\nabla f\circ c_\theta)\cdot h\dtheta. 
$$
The tame product estimate of Lemma~\ref{L485} implies that
$$
\lA f\circ a-f\circ b\rA_{H^{r}} 
\les\int_0^1\Big(\lA \nabla f\circ c_\theta\rA_{H^r}\lA h\rA_{L^\infty} 
+\lA \nabla f\rA_{L^\infty}\lA h\rA_{H^r}\Big)\dtheta. 
$$
This proves the lemma, since
$\lA \nabla f\circ c_\theta\rA_{H^r}\le C(L)\lA f\rA_{H^{r+1}}$ by Lemma~\ref{L486}.
\end{proof}

\subsection{The Cauchy problem and stability}

Theorem~\ref{T:classique} is classical 
(see~\cite{zbMATH03011177,Kato1967,KatoPonce1986,zbMATH01644218,zbMATH00785410}).
Assertion~(2) also follows from Lemma \ref{L487}. Assertion~(3) follows
from the energy estimate 
$\frac{\diff}{\dt}\lA u\rA_{H^{s+1}} \les\lA Du\rA_{L^\infty}\lA u\rA_{H^{s+1}}$, 
from the logarithmic estimate 
$$
\lA Du\rA_{L^\infty}\les\left(1+\lA u\rA_{L^2}+\lA \omega\rA_{L^\infty}\right)
\big(1+\log(2 +\lA u\rA_{H^{s+1}})\big), 
$$
and from the conservation of $\lA u\rA_{L^2}$ and 
$\lA\omega\rA_{L^\infty}$, which give a double exponential bound on
$\lA u(t)\rA_{H^{s+1}}$ depending only on $t$ and $\lA u_0\rA_{H^{s+1}}$ (see~\cite{zbMATH01644218,zbMATH00785410}). 
Assertion~(4) then follows by iterating the continuous dependence given 
by the local theory, using~(3). 

We also record the persistence of higher Sobolev regularity.

\begin{proposition}[Smooth solutions]\label{P204}
Let $u$ be the solution given by Theorem~\ref{T:classique}. 
If $u_0\in H^k(M)$ for all integers 
$k\ge 0$, then $u\in C(\xR;H^k(M))$ for all $k\ge 0$.
\end{proposition}

This follows from Theorem~\ref{T:classique} applied at 
each Sobolev index, together with uniqueness. 

We finally record the Lipschitz stability estimate used in Sections \ref{S:Koch}
and~\ref{S507}. To compare the vorticities in $H^{r-1}$, we interpolate between the velocities and differentiate the
corresponding flows (the product estimate~\e{n261} allows us to do this
without an $L^\infty$ bound on the gradient of the vorticity). 

\begin{lemma}[Lipschitz stability]\label{L489}
Let $T>0$ and let $u,u'$ be two solutions given by Theorem~\ref{T:classique}.
Denote by $\Phi,\Phi'$ their flows and by $\varphi,\varphi'$ the inverse flows, and set 
$$
\mathcal M =\sup_{[0,T]}\left(\lA u(t)\rA_{H^{s + 1}}+\lA u'(t)\rA_{H^{s+1}}\right).
$$
Then, for all real $r\in [1,s]$ and all $t\in [0,T]$, we have 
$$
\lA u(t)-u'(t)\rA_{H^r}+\lA \Phi_t-\Phi'_t\rA_{H^r}
+\lA \varphi_t-\varphi'_t\rA_{H^r}
\le C(T,\mathcal M)\lA u(0)-u'(0)\rA_{H^r},
$$
where $C(T,\mathcal M)$ depends only on $T$, $\mathcal M$, $s$, $r$ and $M$. 
\end{lemma}

\end{document}